\documentclass[11pt,twoside, a4paper, english, reqno]{amsart}
\usepackage{amssymb,amsfonts,latexsym,color}
\usepackage{graphics,verbatim}
\usepackage{graphicx}
\usepackage{a4wide}
\usepackage{color}
\usepackage[normalem]{ulem}

\newtheorem{theorem}{Theorem}[section]
\newtheorem{proposition}{Proposition}[section]
\newtheorem{lemma}{Lemma}[section]
\newtheorem{corollary}{Corollary}[section]
\theoremstyle{Remark}
\newtheorem{remark}{Remark}[section]
\theoremstyle{definition}
\newtheorem{definition}{Definition}[section]

\renewcommand{\>}{\right\rangle}

\newcommand{\eps}{\varepsilon}

\newcommand{\To}{\longrightarrow}

\newcommand{\be} {\begin{equation}}
\newcommand{\ee} {\end{equation}}
\newcommand{\bea} {\begin{eqnarray}}
\newcommand{\eea} {\end{eqnarray}}
\newcommand{\Bea} {\begin{eqnarray*}}
\newcommand{\Eea} {\end{eqnarray*}}
\newcommand{\pa} {\partial}

\newcommand{\al} {\alpha}
\newcommand{\ba} {\beta}
\newcommand{\de} {\delta}
\newcommand{\na}{\nabla}
\newcommand{\ga} {\gamma}
\newcommand{\Ga} {\Gamma}
\newcommand{\Om} {\Omega}
\newcommand{\om} {\omega}
\newcommand{\De} {\Delta}
\newcommand{\la} {\lambda}
\newcommand{\si} {\sigma}
\newcommand{\nequiv} {\not\equiv}
\newcommand{\Si} {\Sigma}

\newcommand{\no} {\nonumber}
\newcommand{\noi} {\noindent}
\newcommand{\lab} {\label}
\newcommand{\va} {\varphi}
\newcommand{\f}{\frac}
\newcommand{\R}{\mathbb R}
\newcommand{\N}{\mathbb N}
\newcommand{\Rn}{\mathbb R^N}

\newcommand{\deb}{\rightharpoonup}
\catcode`\@=11
\newcommand{\Hs}{{\dot{H}^s(\R^N)}}
\newcommand{\HH}{\dot{H}}

\makeatletter \@addtoreset{equation}{section} \makeatother

\begin{document}
\title[Nonlocal scalar field equations]{Nonlocal scalar field equations: qualitative properties, asymptotic profiles and local uniqueness of solutions}

\author{Mousomi Bhakta,  Debangana Mukherjee}
\address{M. Bhakta, Department of Mathematics, Indian Institute of Science Education and Research, Dr. Homi Bhaba Road, Pune-411008, India}
\email{mousomi@iiserpune.ac.in}
\address{D. Mukherjee, Department of Mathematics, Indian Institute of Science Education and Research, Dr. Homi Bhaba Road, Pune-411008, India}
\email{debangana18@gmail.com}

\maketitle
\date{}

\begin{abstract} We study the  nonlocal scalar field equation with a vanishing parameter
$$
\left\{\begin{aligned}
      (-\Delta)^s u+\epsilon u &=|u|^{p-2}u -|u|^{q-2}u \quad\text{in}\quad\mathbb{R}^N \\
      u >0, \quad  u &\in H^s(\mathbb{R}^N),
                \end{aligned}
  \right.
\leqno{(\mathcal{P}_\epsilon)}
$$
where $s\in(0,1)$, $N>2s$, $q>p>2$ are fixed parameters and  $\epsilon>0$ is a vanishing parameter. For $\epsilon>0$ small, we prove the existence of a ground state solution and show that any positive solution of $(\mathcal{P}_{\epsilon})$ is a  classical solution and radially symmetric and symmetric decreasing. We also obtain the decay rate of solution at infinity.  
Next, we study the asymptotic behavior of ground state solutions when $p$ is subcritical, supercritical  or critical Sobolev exponent $2^*=\frac{2N}{N-2s}$. For $p<2^*$, the solution asymptotically coincides with unique positive ground state solution of $(-\Delta)^s u+u=u^p$. On the other hand, for $p=2^*$ the asymptotic behaviour of the solutions is given by the unique positive solution of the nonlocal critical Emden-Fowler type equation. For $p>2^*$, the solution asymptotically coincides with a ground-state solution of $(-\Delta)^s u=u^p-u^q$.
Furthermore, using these asymptotic profile of solutions, we prove  the {\it local uniqueness} of solution in the case $p\leq 2^*$.
\end{abstract}

\tableofcontents

\section{Introduction and main results}
In this paper,  we consider a nonlocal scalar field equation:
\begin{equation}\tag{$\mathcal{P}_\eps$}
  \label{P-eps}
\left\{\begin{aligned}
      (-\De)^s u + \eps u &=|u|^{p-2}u -|u|^{q-2}u \quad\text{in }\quad \Rn, \\
      u&\gneq 0,\\
      u &\in H^s(\Rn), \\
 \end{aligned}
  \right.
\end{equation}
where  $s\in(0,1)$, $N>2s$,  $q>p>2$ and $\eps>0$ is a small parameter often considered in the regime $\eps\to 0$, while
all other parameters are fixed. By $(-\De)^s$ we denote the fractional Laplace operator which could be defined for functions $\varphi$ in the Schwartz class $\mathcal S(\R^N)$ via the Fourier transform as
\begin{equation}\label{d-fracL}
  (\widehat {(-\Delta)^s \varphi}) (\xi)
  = |\xi|^{2s}\widehat{\varphi} (\xi).
\end{equation}
In the local case $s=1$, equation \eqref{P-eps} had been extensively studied.
The existence of a positive radial ground state for sufficiently small $\eps>0$ and all $q>p>2$ goes back to Strauss \cite{S} and Berestycki and Lions \cite{BL}. The uniqueness of the ground states is a result by Serrin and Tang \cite{ST} which uses ODE techniques. More recently, a complete characterisation of the asymptotic profiles of ground states of \eqref{P-eps} as $\eps\to 0$ had been obtained by Moroz and Muratov in \cite{MM}.

One of the main challenges in the study of nonlocal equations is the question of uniqueness of solutions, since the arguments can not rely on the ODE techniques available when $s=1$. In this work we show that asymptotic estimates of the ground state profiles can be used in order to establish a {\em local uniqueness} property of the ground states while in the vicinity of an isolated ground state of a limit equation. The choice of a limit equation for \eqref{P-eps} however is nontrivial and depends on the specific value of $p$ with respect to the critical Sobolev exponent (see \cite{MM}).
Our main goals in this work are:
\begin{itemize}
\item to prove the existence of a positive radially symmetric ground state solution for \eqref{P-eps} for all sufficiently small $\eps>0$ using an adaptation of the Berestycki and P. L. Lions method;

\item to study symmetry property of solution via moving plane method and to determine the decay rate of any positive solution of \eqref{P-eps} at infinity.

\item to identify limit equations and to describe, at least in some cases, asymptotic profiles of the ground sates of \eqref{P-eps} as $\eps\to 0$, extending to the nonlocal case the results in \cite{MM}.

\item to use asymptotic properties of the ground states in order to establish the local uniqueness of the ground states of \eqref{P-eps} as $\eps\to 0$.
\end{itemize}

\bigskip

Asymptotic behavior of the ground states $u_{\eps}$ naturally arises in the study of various bifurcation problems, for which \eqref{P-eps} can
be considered as a canonical normal form (see e.g. \cite{CH, vH}). As mentioned in \cite{MM}, problem  \eqref{P-eps} itself may also be considered as a prototypical example of a bifurcation problem for elliptic equations. In fact, our results are expected to remain valid for a broader class of scalar field equations whose nonlinearity has the leading terms in the expansion around zero which coincides with the ones in  \eqref{P-eps}. It is known that for $s=1$, problem  \eqref{P-eps} appears in the study of nonclassical nucleation near spinodal in mesoscopic models of phase transitions \cite{CaH, MV}, as well as in the study of the decay of false vacuum in quantum field theories \cite{C}. In the case of $\eps=0$ and $s\in (0,1)$, existence and regularity properties of \eqref{P-eps} have been studied in \cite{BM}. 

\bigskip

Our first main result reads as follows.

\begin{theorem}\label{t:BL}
Let $N>2s$ and $q>p>2$. Then there exists $\eps_*>0$ such that for all $\eps\in(0,\eps_*)$ the problem \eqref{P-eps} admits a ground--state solution $u_\eps\in H^s(\R^N)$.
Moreover, $u_\eps$ is a radially symmetric and decreasing H\"older continuous function of $|x|$.
In addition, $0<u_\eps(x)\le 1$ for all $x\in\R^N$ and
\begin{equation}
u_\eps(x)= C_\eps|x|^{-(N+2s)}+o(|x|^{-(N+2s)})\quad\text{as }\,|x|\to\infty,
\end{equation}
where $C_\eps>0$ depends on $N, s, p, q, \eps$.
\end{theorem}

Our next result concerns the radial symmetry and decreasing property of weak solutions of \eqref{P-eps}.

\begin{theorem}[Radial symmetry]\lab{t:rad}
Let $u\in H^s(\R^N)$ be a weak solution of \eqref{P-eps}.	
Then $u$ is radially symmetric and strictly decreasing about some point in $\Rn$.
\end{theorem}

We prove the symmetry result in the spirit of \cite[Theorem 1.6]{DMPS} and  \cite[Theorem 1.2]{FW}. The main difference of our theorem with their results is, we have not assumed any apriori decay rate of the solution,
while our result includes all the three types of power nonlinearities namely, subcritical, critical and supercritical nonlinearities. One of the first hurdle to prove this theorem is to show that $w_\lambda$ (see the definition \ref{+}), which is anti-symmetric w.r.t. the reflection hyper-plane,  belongs to $H^s(\Rn)$. We prove this by decomposing $w_\lambda$ into two suitable parts and using the regularity properties established in Section 2.2 and the fractional Hardy inequality for the half-space (see Lemma \ref{l-test}). 

Next, we are concerned with the asymptotic behavior of ground-state solutions of \eqref{P-eps} as $\eps\to 0$ in the three different cases, namely $p=2^*$, $2<p<2^*$ and $p>2^*$. 

Studying the asymptotic profile of solutions for nonlocal elliptic equations started very recently. For instance, multi-peak solutions of a fractional Schr\"odinger equation in the whole of $\R^N$ was considered  in \cite{DPW}. To see the study 
of asymptotic behavior of solution for the equations of the type
\begin{equation}
\left\{\begin{aligned}
      \eps^{2s} (-\De)^s u &= f(u)   \quad\text{in } \Om,\\
    u &>0 \quad \text{in} \ \Omega,  \\
      u &= 0 \quad\text{on } \R^N \setminus  \Om,
    \end{aligned}
  \right.
\end{equation}
where $f$ is having superlinear nonlinearity with $f(0)=0$, $\Om $ is a smooth bounded domain in $\mathbb{R}^N$, we refer \cite{BMS, DPDV, DRS} and the references there-in. In \cite{KCL}, asymptotic profile of solution for the equations with spectral fractional Laplacian have been studied. 

\subsection{Asymptotic behavior and local uniqueness }
\subsection{Critical case $p=2^*$}
It is well-known that the radial ground states of the equation
\begin{align}\label{R-*}
(-\De)^s u=u^{2^*-1} \quad\mbox{in} \quad \Rn,
\end{align}
are given by the function
\begin{align}\label{U-1}
U_1(x):=c_{N,s}\big(1+|x|^2\big)^{-\big(\f{N-2s}{2}\big)},
\end{align}
where \be\lab{c-ns}
c_{N,s}=2^\f{N-2s}{2}\displaystyle\left(\f{\Ga(\f{N+2s}{2})}{\Ga(\f{N-2s}{2})}\right)^\f{N-2s}{4s}.\ee
We consider the family of its rescalings
\begin{align}
U_\la(x):=\la^{-\f{N-2s}{2}}U_1\bigg(\f{x}{\la}\bigg).
\end{align}

\begin{theorem}\label{thm-2.5}
	Let $p=2^*$ and $N>4s$ and $u_\eps$ be a ground-state solution of \eqref{P-eps}. Then, there exists a rescaling $\la_\eps : (0,\eps_*) \to (0,\infty)$ such that as $\eps \to 0,$ the rescaled family of ground states
	\begin{equation}\label{tilde-v-eps}
	\widetilde{v_\eps}(x):=\la_\eps^{\f{(N-2s)}{2}}u_\eps(\la_{\eps}x)
	\end{equation}
	converges to $U_1(x)$ in $\dot{H}^s(\Rn), L^q(\Rn)$ and $C^{2s-\de}(\Rn)$ for some $\de\in(0, 2s)$. Furthermore,
	\begin{equation}\label{la-est}
	\la_\eps \sim \eps^{-\f{(p-2)}{2s(q-2)}}, 
	\end{equation} and
	\begin{equation}\label{u-eps-est}
	u_\eps(0) \sim \eps^{\f{1}{q-2}}. 
	\end{equation}
\end{theorem}

\begin{theorem}\lab{t:uni-cri}
	Let $p=2^*$ and $N>4s$. Suppose there are two sequences of ground state solutions $u_{\eps}^1$ and $u_{\eps}^2$ of \eqref{P-eps} such that
	$$\|\widetilde{v}_{\eps}^i-U_1\|_{\dot{H}^s(\Rn)\cap C^{2s-\si}(\Rn)}\to 0, \quad\text{as}\quad \eps\to 0, \quad i=1, 2,$$ for some $\si\in(0,2s)$,
	where $\widetilde{v}_\eps^i$ are defined by \eqref{tilde-v-eps} and $U_1$ is as defined in  \eqref{U-1}. In addition, if $u_{\eps}^i$, $i=1,2$, are in $H^{2s}(\Rn)$, then there exists $\eps_0>0$ such that for any $\eps\in(0,\eps_0)$, we have $u_{\eps}^1=u_{\eps}^2$.
\end{theorem}

\subsection{Supercritical case $p>2^*$}

For $p>2^*$, the limit equation
\begin{equation}
\label{P_0}
(P_{0})\qquad
(-\De)^s u  =|u|^{p-2}u -|u|^{q-2}u \quad\text{in }\quad \Rn, 
\end{equation}
admits  a non-negative radially decreasing solution $u_0 \in \dot{H}^s(\Rn) \cap L^q(\Rn)$ (see \cite{BM}, theorem 1.7]). Further, following the similar arguments of Lemma \ref{l:max},  we obtain  $0\leq u_0\leq 1$. Therefore, since $u_0$ is a classical solution, using maximum principle it follows that $u_0$ is strictly positive in $\Rn$. Also, from [\cite{BM}, Theorem 1.3, Theorem 1.4] we have, $u_0 \in C^{2s+\alpha}(\Rn)$ and 
\begin{equation*}
\lim_{|x| \to \infty}|x|^{N-2s}u_0(x)=c_0,
\end{equation*}
for some $c_0>0$.

 \begin{theorem}\label{t:supcri}
 	Let $N>2s$ and $q>p>2^*$.  Suppose $(u_\eps)_{\eps>0}$ is a family of ground state of \eqref{P-eps}. Then,
 	there exists a ground state solution $u_0$ of $(P_0)$ such that $u_\eps \to u_0$ in  $\dot{H}^s(\Rn)$, $L^{q}(\Rn)$ and $C^{2s-\alpha}(\Rn)$, for some $\al\in(0,2s)$.
 \end{theorem}	
 
 \begin{remark}
 In the local case $s=1$, it is known that $(P_0)$ has a unique solution (see \cite{KMPT}) when $q>p>2^*$, where as for $s\in (0,1)$ the uniqueness of $(P_0)$ is not yet known. Due to this reason, we are not able to establish exact asymptotic behavior of $u_\eps$ as it was done in \cite{MM}. It is worth mentioning that for the lower dimensional case $(2s<N\leq 4s)$ involving fractional Laplace operator, the energy and norm-estimates are very delicate and technically involved to compute (see Section 6.1).  
 \end{remark}

\subsection{Subcritical case $2<p<2^*$}

\vspace{2mm}

Proceeding as in the proof of Lemma \ref{l:2}, it can be proved that solution of
$(P_0)$ are in $L^{\infty}(\Rn)$. Consequently,  Lemma \ref{l:Phz} implies $(P_0)$ has no solution in the subcritical case. Therefore, in view of Lemma \ref{l:2}, the family of ground state solution $u_{\eps}$ must converge to $0$, on compact subsets of $\Rn$. To describe, asymptotic behavior of $u_{\eps}$, we use a canonical rescaling:
\begin{align}\label{v}
v(x)=\eps^{-\f{1}{s(p-2)}}u(\eps^{\f{-1}{2s}}x), \quad\mbox{for all}\quad x \in \Rn,
\end{align}
then ($P_{\eps}$) transforms into the equation

\begin{equation} \label{R-eps}
(\tilde{P}_{\eps})\hspace{4mm}
(-\De)^s v +v =|v|^{p-2}v-\eps^{\f{(q-2)}{s(p-2)}-1}|v|^{q-2}v \quad\mbox{in}\quad \Rn.
\end{equation}

Note that $s\in(0,1)$ implies $p>2$ if and only if $p>s(p-2)+2$ and therefore, $q>s(p-2)+2$. Hence, the limit problem associated to $(\tilde{P}_{\eps})$ as $\eps \to 0$ has the form
\begin{align}\label{R-0}
(-\De)^s v +v =|v|^{p-2}v \quad\mbox{in}\quad \Rn.
\end{align}

In the subcritical case $2<p<2^*$, it is known that Eq.\eqref{R-0} admits a unique radial ground state solution $v_0$. Existence part is proved in \cite{DPV} and for uniqueness, see \cite{FLS}. It has also been proved in \cite{FLS} that $v_0$ is positive and  a decreasing function of $|x|$ and by \cite[Theorem 1.5]{FQT} we have
$$C_1|x|^{-(N+2s)}\leq v_0(x)\leq C_2|x|^{-(N+2s)}.$$
On the other hand,  when $p \geq 2^*$, Eq.\eqref{R-0} has no nontrivial finite energy solutions, which is a direct consequence of Poho\v zaev's identity \cite{CW}.

Rescaling back to the original function, we prove the following result:
\begin{theorem}\lab{t:sub}
	Let $2<p<2^*$, $N>2s$, $q>p$ and $u_\eps$ be a ground-state solution of \eqref{P-eps}. Then, as $\eps\to 0$, the rescaled family of ground states
	$$ v_{\eps}(x)=\eps^{-\f{1}{s(p-2)}}u_{\eps}(\eps^{\f{-1}{2s}}x)$$
	converge to $v_0$ in $H^s(\Rn)$, $L^{q}(\Rn)$ and $C^{2s-\sigma}(\Rn)$, for some $\si\in(0,2s)$. In particular,
	$$u_{\eps}(0)\simeq \eps^{\f{1}{s(p-2)}}v_0(0).$$
\end{theorem}

\begin{theorem}\lab{t:uni-sub}
	Let $s\in (0, 1)$, $N>2s$ and $2<p<2^*$.  Suppose there are two sequences of ground state solutions $u_{\eps}^1$ and $u_{\eps}^2$ of \eqref{P-eps} such that
	$$\|v_{\eps}^i-v_0\|_{H^s(\Rn)\cap C^{2s-\si}(\Rn)}\to 0, \quad\text{as}\quad \eps\to 0, \quad i=1, 2,$$ for some $\si\in(0,2s)$,
	where $v_\eps^i$ are defined by \eqref{v} and $v_0$ is unique ground state solution of \eqref{R-0}. Then there exists $\eps_0>0$ such that for any $\eps\in(0,\eps_0)$, we have $u_{\eps}^1=u_{\eps}^2$.
\end{theorem}

\begin{remark}
Using the asymptotic behavior of the ground state solutions, studying the local uniqueness of the solutions in the critical and subcritical cases for \eqref{P-eps}, namely Theorem \ref{t:uni-cri} and Theorem \ref{t:uni-sub} are new even in the local case $s=1$.
\end{remark}

 The rest of the paper is organised as follows. Section 2 deals with the preliminaries and we discuss some qualitative properties of the solutions of \eqref{P-eps}. Section 3 is devoted to the proof of radial symmetry of any weak solution of \eqref{P-eps}. In Section 4, we prove Theorem \ref{t:BL}. In Section 5, we give the proof of Theorem \ref{thm-2.5}. Section 6 contains the proof of Theorem \ref{t:supcri}.  In Section 7, we prove Theorem \ref{t:sub}. Section 8 and Section 9 are devoted to the proof of Theorem \ref{t:uni-sub} and Theorem \ref{t:uni-cri} respectively.

{\bf Notations}:

For $\eps\ll 1$ and $f(\eps),\, g(\eps) \geq 0,$ whenever there exists $\eps_0>0$ such that for every $0<\eps \leq \eps_0$ the respective
condition holds, we write:
\begin{itemize}
	\item
	$f(\eps)\lesssim g(\eps)$ if there exists $C>0$ independent of $\eps$ such that $f(\eps) \leq Cg(\eps);$
	\item
	$f(\eps)\sim g(\eps)$ if $f(\eps) \lesssim g(\eps)$ and $g(\eps) \lesssim f(\eps);$
	\item
	$f(\eps) \simeq g(\eps)$ if $f(\eps) \sim g(\eps)$ and $\lim_{\eps \to 0}\f{f(\eps)}{g(\eps)}=1.$
\end{itemize}
We denote by $\|.\|_p$, the $L^p$ norm in $\Rn$.
We also use the standard notations $f=O(g)$ and $f=o(g),$ where $f \geq 0, g \geq 0$. $C, c$ denote generic positive constants independent
of $\eps$ and it may vary from line to line.

\section{Preliminaries}

We rewrite \eqref{P-eps} in the form
$$
(-\De)^s u =f_\eps(u),\quad u\gneq 0\qquad\text{in }\, H^s(\R^N),
$$
where
$$f_\eps(u):=-\eps u+u^{p-1}-u^{q-1},$$
and $F_\eps(u)=\displaystyle\int_0^u f_\eps(\zeta)d\zeta$. In what follows we always assume that $s\in(0,1)$, $N>2s$, $q>p>2$ and $\eps>0$.

\subsection{Fractional Sobolev spaces}
Recall that for $N>2s$ the {\em homogeneous Sobolev space} $\HH^s(\R^N)$ can be defined as the completion of $C^\infty_c(\R^n)$ with respect to the norm
\begin{equation}\label{22-1}
\|u\|_{\dot{H}^s(\Rn)}^2:=\int_{\Rn}|\xi|^{2s}|\hat{u}(\xi)|^{2}\,d\xi<\infty.
\end{equation}
If $u\in\HH^s(\R^N)$ then $u\in L^2_{loc}(\R^N)$ and if $s\in(0,1)$, then $\HH^s(\R^N)$--norm admits Gagliardo representation
\begin{equation}\label{e-gagliardo}
\|u\|_{\dot{H}^s(\Rn)}^2=c_{N,s}\iint_{\Rn\times\Rn}\frac{|u(x)-u(y)|^2}{|x-y|^{N+2s}}dy\,dx,
\end{equation}
for some $c_{N,s}>0$ \cite[Proposition 1.37]{Bahouri}.
The fractional {\em Sobolev inequality} states that there exists  a positive constant $S=S(N,s)$ such that
$$\|u\|_{\dot{H}^s(\Rn)}^2\ge S\|u\|^2_{L^{2^*}(\Rn)}\qquad\forall\,u\in \HH^s(\Rn),$$
where $2^*:=\f{2N}{N-2s}$ is the critical Sobolev exponent (cf. \cite[Section 1.3.2]{Bahouri}).
In particular, $\HH^s(\R^N)$ is a well-defined space of functions and $\HH^s(\R^N)\subset L^{2^*}(\R^N)$.

The space $\HH^s(\R^N)$ is a Hilbert space with the scalar product
$$\langle u,v\rangle_{\HH^s(\R^N)}:=\int_{\Rn}|\xi|^{2s}\hat{u}(\xi)\hat{v}(\xi)\,d\xi=
\int_{\Rn}(\widehat {(-\Delta)^s u}) (\xi)\hat{v}(\xi)\,d\xi.$$
Recall also that if $s\in(0,1]$ then the space $\HH^s(\R^N)$ is a {\em Dirichlet space}, that is
$u\in\HH^s(\R^N)$
implies that $u^+\wedge 1\in \HH^s(\R^N)$ and $\|u^+\wedge 1\|_{\HH^s(\R^N)}\le\|u\|_{\HH^s(\R^N)}$, see \cite[p.34 and p.43]{MaR}.
These contraction invariance properties are sufficient in order to establish weak maximum and comparison principles for $(-\De)^s$, see Lemma \ref{l:max} below.
Here and in the sequel we denote $u^+:=\max\{u,0\}$ and $u^-:=\max\{-u,0\}$, so that $u=u^+-u^-$.

The energy associated with problem \eqref{P-eps} takes the form
\begin{multline}
E_\eps(u):=\f{1}{2}\|u\|_{\HH^s(\R^N)}^2+\f{\eps}{2}\|u\|^2_{L^{2}(\Rn)}
-\f{1}{p}\|u\|^p_{L^{p}(\Rn)}+\frac{1}{q}\|u\|^q_{L^{q}(\Rn)}\\
=\f{1}{2}\|u\|_{\HH^s(\R^N)}^2+\int_{\R^N}F_\eps(u)dx,
\end{multline}
where $F_\eps$ is the primitive of $f_\eps$. Since we do not restrict the values of $q$ to a subcritical range, the natural domain of definition for $E_\eps$
is $H^s(\R^N)\cap L^q(\R^N)$, where $H^s(\R^N)$ is the {\em nonhomogneous Sobolev space} which can be defined for $s<N/2$ as the subspace of $\HH^s(\R^N)$ such that
\begin{equation}\label{22-2}
\|u\|_{H^s(\Rn)}^2:=\|u\|_{\dot{H}^s(\Rn)}+\|u\|_{L^{2}(\Rn)}^2<\infty.
\end{equation}
Note that $H^s(\R^N)\cap L^q(\R^N)\subset L^p(\R^N)$ by interpolation, since $2<p<q$.

\begin{definition}\lab{def-1}
We say that $u\in H^s(\Rn)\cap L^q(\R^N)$ is a {\em weak solution} of \eqref{P-eps} if
$$
\langle u,\varphi\rangle_{\HH^s(\R^N)}+\eps \int_{\Rn}u\va\ dx
=\int_{\Rn}u^{p-1}\va\ dx-\int_{\Rn}u^{q-1}\va\ dx\quad\forall\va\in H^s(\Rn)\cap L^q(\R^N).$$
\end{definition}

Recall that the fractional Laplace operator \eqref{d-fracL} admits an integral representation
\begin{align} \label{De-u}
\left(-\Delta\right)^s u(x)=-\frac{c_{N,s}}{2}\int_{\mathbb{R}^N}\frac{u(x+y)-2u(x)+u(x-y)}{|y|^{N+2s}}dy, \quad x\in\Rn,
\end{align}
where $c_{N,s}:=\f{2^{2s}s\Ga(N/2+s)}{\pi^{N/2}\Ga(1-s)}$.
Because of the strong singularity at the origin, this expression should be interpreted as a singular integral.
However, if the function $u$ is sufficiently regular, e.g. if $u\in L^1\big(\R^N,\frac{dx}{(1+|x|)^{(N+2s)}}\big)\cap C^{2s+\alpha}(\R^N)$ for some $\alpha>0$,
then for every $x\in\R^N$ the right hand side of \eqref{De-u} is well defined and finite in the sense of Lebesgue's integration \cite[Proposition 3.1]{BM}.
This suggests the following definition.

\begin{definition}
A function $u:\R^N\to\R$ is a {\em classical solution} of \eqref{P-eps} if 

$u\in L^1\big(\R^N,\frac{dx}{(1+|x|)^{(N+2s)}}\big)\cap C^{2s+\alpha}(\R^N)$ for some $\alpha>0$ and
\be\lab{u-f}
(-\De)^s u(x)=f_\eps(u(x))\qquad\forall\, x\in\Rn.
\ee
\end{definition}

Next we are going to show that every weak solution of \eqref{P-eps} is in fact a classical solution
and deduce some other properties of weak solutions of \eqref{P-eps}.

\subsection{Qualitative properties of weak solutions of \eqref{P-eps}.}\label{s-qual}
We first describe some apriori qualitative properties of weak solutions of \eqref{P-eps}.
First we establish an apriori bound which is a consequence of the weak maximum principle for $(-\Delta)^s$ with $s\le 1$.

\begin{lemma}[Apriori bound]\label{l:max}
Let $u\in H^s(\R^N)\cap L^q(\R^N)$ be a weak solution of \eqref{P-eps}.
Then $u\le 1$ a.e. in $\R^N$.
\end{lemma}

\proof
Recall that the space $\dot{H}^s(\R^N)$ is invariant with respect to standard truncations and hence $u\in \dot{H}^s(\R^N)$ can be represented as $u=(u\wedge 1)+(u-1)^+$, where $u\wedge 1,(u-1)^+\in \dot{H}^s(\R^N)$, see \cite[p.34]{MaR}.\footnote{Here $u\wedge w:=\min\{u,w\}$.}
In view of the contraction properties of Dirichlet forms, for all $u\in \dot{H}^s(\R^N)$ the following inequality holds
$$\langle (u\wedge 1),(u-1)^+\rangle_{\dot{H}^s}\ge 0,$$
see \cite[p.32 and 34]{MaR}.
Therefore,
$$\langle u,(u-1)^+\rangle_{\dot{H}^s}=\langle (u\wedge 1),(u-1)^+\rangle_{\dot{H}^s}+\langle(u-1)^+,(u-1)^+\rangle_{\dot{H}^s}\ge\|(u-1)^+\|_{\dot{H}^s}^2.$$
On the other hand, testing equation \eqref{P-eps} against $(u-1)^+$, we obtain
$$\langle u,(u-1)^+\rangle_{\dot{H}^s}\le\int_{\R^N}f_\eps(u)(u-1)^+\le 0,$$
since $f_\eps(u(x))\le 0$ for all $x\in \mathrm{Supp}((u-1)^+)$. We conclude that $\|(u-1)^+\|_{\dot{H}^s}^2\le 0$.
\qed

\begin{lemma}[Regularity and positivity]\label{l:2}
Let $u\in H^s(\R^N)\cap L^q(\R^N)$ be a weak solution of \eqref{P-eps}. Then 

(i) $u$ is a classical solution of \eqref{P-eps}. Moreover, $0<u(x)<1$ for all $x\in\R^N$ and $u(x) \to 0$ as $|x| \to \infty$.

(ii) $u\in C^{\infty}(\Rn)$  if both $p$ and $q$ are integer and $u\in C^{2ks+2s}(\Rn)$, where $k$ is the largest integer satisfying   $\lfloor 2ks\rfloor<p$  if $p\not\in\N$ and $\lfloor 2ks\rfloor<q$ if $p\in\N$ but $q\not\in\N$, where $\lfloor 2ks\rfloor$ denotes the greatest integer less than equal to $2ks$ .
\end{lemma}

\begin{proof}
(i)	From Lemma \ref{l:max}, we conclude that $f_\eps(u):=u^{p-1}-u^{q-1}-\eps u\in L^{\infty}(\Rn)$. Therefore applying Schauder estimate \cite[Theorem 1.1 (a)]{RS1}, $u\in C^{2s}(B_{1/2}(0))$ when $s\not=\f{1}{2}$ and in $C^{2s-\de}((B_{1/2}(0)))$, when $s=\f{1}{2}$. Moreover, since the equation is invariant under translation, translating the equation as in the proof of \cite[Theorem 1.4]{BM}, we obtain $u\in C^{2s}(\Rn)$ when $s\not=\f{1}{2}$ and in $C^{2s-\de}(\Rn)$, when $s=\f{1}{2}$. Hence, $f_\eps(u)\in C^{2s}(\Rn)$ when  $s\not=\f{1}{2}$ and in $C^{2s-\de}(\Rn)$, when $s=\f{1}{2}$. Thus, invoking \cite[Theorem 1.1 (b)]{RS1} we have $u\in C_{loc}^{2s+\alpha}(\Rn)$ for some $\al\in(0,1)$. Next, it can be shown exactly as in 	\cite[Proposition 1]{BM} that $u$ is a classical solution of $(P_{\eps})$.
	
	In particular, using integral representation \eqref{De-u} for $(-\Delta)^s$, we conclude that
	either $u(x)>0$ for all $x\in\R^N$, or otherwise $u\equiv 0$.
	Similarly, if there exists $x\in\Rn$ such that $u(x)=1$, then $x$ is the maximum point of $u$. Since $u$ is a classical solution, using the
	integral representation \eqref{De-u}, we have $(-\De)^su(x)>0$. Therefore, LHS of \eqref{P-eps} is strictly positive at $x$, whereas RHS of \eqref{P-eps} equal to $0$ at $x$, which is a contradiction. Hence $0<u<1$.

	Finally, it is standard to see that $u\in C^{\alpha}(\Rn)\cap
	L^2(\Rn)$ for some $\alpha\in(0,1]$ implies that
	$u(x) \to 0$ as $|x| \to \infty$.
	
	(ii) Since the equation is invariant under translation, repeating the argument as in the first part of (i) (see also \cite[Theorem 1.4]{BM}), we can improve the regularity $C^{\infty}(\Rn)$  if both $p$ and $q$ are integer and $C^{2ks+2s}(\Rn)$, where $k$ is the largest integer satisfying   $\lfloor 2ks\rfloor<p$  if $p\not\in\N$ and $\lfloor 2ks\rfloor<q$ if $p\in\N$ but $q\not\in\N$, where $\lfloor 2ks\rfloor$ denotes the greatest integer less than equal to $2ks$.
\end{proof}

Before proving the next lemma we recall an important result from Frank-Lenzmann-Silvestre \cite[Lemma C.2]{FLS} which is a key tool in proving the decay estimate of the  solution of $(P_{\eps})$.

\begin{theorem}\label{t:FLS}\cite[Lemma C.2]{FLS}
	Let $N>1$ and $0<s<1$, and suppose that $V\in L^{\infty}(\Rn)$ with
	$V(x)\to 0$ as $|x|\to\infty$. Assume that $u\in L^2(\Rn)$ with $\|u\|_{2}=1$
	and  satisfies \\
	$(-\De)^s u+Vu=E u$ with some $E<0$. Furthermore, let $0<\la<-E$ be given and suppose that $R>0$ is such that $V(x)+\la\geq 0$ for $|x|\geq R$. Then the following properties hold:\smallskip
	
	(i) $|u(x)|\lesssim(1+|x|)^{-(N+2s)},$
	
	(ii) $u(x)= -c\la^{-2}\bigg(\displaystyle\int_{\Rn}Vu\, dx\bigg)|x|^{-(N+2s)}+o(|x|^{-(N+2s)})$  as $|x|\to\infty$,\\
where $c>0$ is a positive constant depending on $N$ and $s$.
\end{theorem}

\begin{lemma}[Asymptotic]\lab{t:decay}
	Let $u\in H^s(\R^N)\cap L^q(\R^N)$ be a weak solution of \eqref{P-eps}.
	Then
	$$u(x)= C|x|^{-(N+2s)}+o(|x|^{-(N+2s)})\quad\text{as }\,|x|\to\infty,$$

	where $C>0$ depends on $N, s, p, q, \eps$ and $u$.
\end{lemma}
\begin{proof}
	We first obtain an upper estimate for $u$. To do this, define $\tilde{u}=\f{u}{\|u\|_{2}}$. Then $\tilde{u}$ satisfy

	\be\lab{5-29-1}
	(-\De)^s \tilde{u}+V\tilde{u}=-\eps\tilde{u},
	\ee
	where $$V(x)=\|u\|_2^{q-2}\tilde{u}^{q-2}(x)-\|u\|_2^{p-2}\tilde{u}^{p-2}(x)=u^{q-2}(x)-u^{p-2}(x).$$  Therefore, by Lemma \ref{l:max}, we have $V\in L^{\infty}(\Rn)$. Further, by Lemma \ref{l:2}, it follows that $V(x)\to 0$ as $|x|\to\infty$ and for given $0<\la<\eps$, there exists $R>0$ such that $V(x)+\la\geq 0$ for $|x|\geq R$. Hence, applying Theorem \ref{t:FLS}(i), we conclude that
	\begin{equation}\label{e-up}
	\tilde u(x)\leq c_1(1+|x|)^{-(N+2s)}.
	\end{equation}
	
	To obtain the asymptotic estimate from Theorem \ref{t:FLS}(ii) it is enough if we show $$-c\la^{-2}\bigg(\displaystyle\int_{\Rn}V\tilde{u}\, dx\bigg)>c_2,$$ for some constant $c_2>0$.
	To see this, we note that since $u\le 1$ by Lemma \ref{l:max}, and since $p<q$, we have $u^{p-1}-u^{q-1}\gneq 0$. Moreover, upper estimate \eqref{e-up} implies that  $u\in L^{p-1}(\Rn)\cap L^{q-1}(\Rn)$ as $q>p>2>\f{2(N+s)}{N+2s}$. Hence,
	$$-C\la^{-2}\bigg(\displaystyle\int_{\Rn}V\tilde{u}\, dx\bigg)=\f{C}{\la^2\|u\|_2}\int_{\Rn}(u^{p-1}-u^{q-1}) dx>0.$$
	Thus the assertion follows.
\end{proof}

\begin{lemma}[Poho\v zaev identity]\label{l:Phz}
	Let $u\in H^s(\R^N)\cap L^q(\R^N)$ be a weak solution of \eqref{P-eps}. Then $u$ satisfies Poho\v zaev identity
	$$\|u\|_{\dot{H}^s(\R^N)}^2=\frac{2N}{N-2s}\int_{\R^N}F_\eps(u).$$
\end{lemma}

\proof
See \cite[Theorem A.1]{BM}, where Poho\v zaev identity is proved for weak solutions $u\in \dot{H}^s(\R^N)\cap L^\infty(\R^N)$.
\qed

We conclude this section by stating  the following decay properties of radial decreasing functions on $\Rn.$

\begin{lemma}\label{l:1}
	(i)  Let $t \geq 1$ and $u \in L^t(\Rn)$ be a radial nonincreasing function. Then for every $x \neq 0$,
	$$u(x) \leq C(N,t)|x|^{\f{-N}{t}}\|u\|_t,$$
	where $C(N,t)=|B_1(0)|^{-1/t}$
	
	(ii) Let $u_n \in H^s(\Rn)$ be a sequence of radial nonincreasing functions such that $u_n \rightharpoonup u$ in $H^s(\Rn).$ Then
	upon extracting a subsequence,
	$u_n \to u$ in $L^{\infty}(\Rn \setminus B_r(0))$ and $L^p(\Rn \setminus B_r(0))$ for all $r>0,$ for all $p>2^*.$
\end{lemma}
\begin{proof}
	See \cite[Lemma A.IV]{BL} for (i).  On the other hand, (ii) follows from \cite[lemma 6.1]{BM}.
\end{proof}

\vspace{3mm}

\section{Radial symmetry of weak solutions}
In this section we prove the Theorem \ref{t:rad}. Let $u\in H^s(\R^N)$ be a weak solution of \eqref{P-eps}. Recall that by the results in the previous section, $u$ is a classical solution of \eqref{P-eps}. In particular, $u$ is H\"older continuous and $u(x)\to 0$ as $|x|\to\infty$.

Before we proceed to the proof of Theorem \ref{t:rad},
we establish some auxiliary results about the properties of the reflections of $u$.
	
For $\la<0$, we define $$\Si_{\la}=\{x\in\Rn : x_1<\la \},$$ and for $x\in\Si_{\la}$,
let $x_{\la}$ denote the it's reflection to the hyperplane $x_1=\la$, that is $x_{\la}=(2\la-x_1, x_2, \cdots, x_n)$. Set
$$u_{\la}(x):=u(x_{\la}), \quad x\in\Si_{\la}$$
and note that, in view of translation invariance, $u_\la$ is also a weak solution of  \eqref{P-eps}.
Next, define
	\begin{equation}
	\label{+}w_{\lambda}(x):=
	\begin{cases}
	(u-u_\la)^+(x),&\quad x \in \Sigma_\la, \\
	-(u-u_\la)^-(x),&\quad x \in \Sigma_\la^c, \\
	\end{cases}
	\end{equation}
so that $w_\lambda$ is antisymmetric with respect to $T_\lambda$.

\begin{lemma}\label{l-test}
	$w_{\la}\in H^s(\R^N).$
\end{lemma}

\proof

Note that, $w_\la$ can be rewritten as $w_{\la}=(w_{\la}^1)^+-(w_{\la}^2)^-$, where
	\begin{equation*}
	w_{\lambda}^1(x):=
	\begin{cases}
	(u-u_\la)(x),&\quad x \in \Sigma_\la, \\
	0,&\quad x \in \Sigma_\la^c, \\
	\end{cases}
	\end{equation*}
	and
\begin{equation*}
	w_{\lambda}^2(x):=
	\begin{cases}
	0,&\quad x \in \Sigma_\la, \\
	(u-u_\la)(x),&\quad x \in \Sigma_\la^c, \\
	\end{cases}
	\end{equation*}

\noindent
{\bf Claim:} {\em $w_{\la}^1, \, w_{\la}^2\in H^s(\R^N)\cap C^{s+\alpha}(\R^N)$, for some $\alpha>0$.}

To prove the claim, we observe that by Lemma \ref{l:2}, $u-u_{\la}\in H^s(\R^N)\cap C^{2s+\alpha}(\R^N)$, for some $\alpha>0$.
Since $u_\la(x)=u(x)$ for $x_1=\lambda$, we have $u-u_{\la}\in C^{\sigma}(\R^N)$, where $\sigma:=\min\{1,2s+\alpha\}$. From this, it is easy to see that $w_{\la}^1\in C^{s+\alpha}(\Rn)$, for some $\al>0$. Indeed, let $x\in\bar\Sigma_\lambda$, $y\in\bar\Sigma_\lambda^c$ and let $y'\in\partial\Sigma_\lambda$ be the intersection of the straight line interval $[x,y]$
with $\partial\Sigma_\lambda$. Then
$$|w_{\la}^1(x)-w_{\la}^1(y)|=|w_{\la}^1(x)-w_{\la}^1(y')|\le C|x-y'|^\sigma\le C|x-y|^\sigma.$$ As for $\alpha>0$ small enough, ${s+\alpha}<\sigma$, we conclude that $w_{\la}^1\in C^{s+\alpha}(\Rn)$.

Since $|w_\la^1|\le|u-u_\la|$, it is clear that $w_\la^1\in L^2(\R^N)$. Thus, in order to prove that $w_{\la}^1\in H^s(\R^N)$,
it is sufficient to show that Gagliardo norm \eqref{e-gagliardo} of $w_{\la}^1$ is finite, see \cite[Proposition 1.59]{Bahouri}.
Since $w_\la^1=0$ in $\Sigma_\la^c$, we can write
\begin{multline*}
||w_{\la}^1||_{\dot{H}^s(\Rn)}^2=\int_{y\in\Sigma_\la}\int_{x\in\Sigma_\la}\frac{|w_{\la}^1(x)-w_{\la}^1(y)|^2}{|x-y|^{N+2s}}dx\,dy\\
+2\int_{y\in\Sigma_\la^c}\int_{x\in\Sigma_\la}\frac{|w_{\la}^1(x)-w_{\la}^1(y)|^2}{|x-y|^{N+2s}}dx\,dy=: I_1+I_2,
\end{multline*}
where $I_1\le\|u-u_\lambda\|_{\dot{H}^s(\Rn)}^2$ since $w_\la^1=u-u_\la$ on $\Sigma_\la$. Also, we obtain 
\begin{eqnarray}
	I_2&=&2\int_{y\in\Sigma_\la^c}\int_{x\in\Sigma_\la}\frac{|w_{\la}^1(x)-w_{\la}^1(y)|^2}{|x-y|^{N+2s}}dx\,dy=2\int_{y\in\Sigma_\la^c}\int_{x\in\Sigma_\la}\frac{|w_{\la}^1(x)|^2}{|x-y|^{N+2s}}dx\,dy\no\\
	&=&\int_{x\in\Sigma_\la} |w_{\la}^1(x)|^2\int_{y\in\Sigma_\la^c}\frac{1}{|x-y|^{N+2s}}dy\,dx\no\\
	&\le&c_1\int_{x\in\Sigma_\la} \frac{|w_{\la}^1(x)|^2}{|x_1-\lambda|^{2s}}\,dx\le c_2\|u-u_\lambda\|_{\dot{H}^s(\Rn)}^2,
\end{eqnarray}
by the fractional Hardy inequality in the half--space \cite{FS-half}. Hence $w_{\la}^1\in H^s(\Rn)\cap C^{s+\al}(\Rn)$. Similarly the claim can be proved for $w_{\la}^2$.

Therefore, we can conclude that $(w_{\la}^1)^+, \, (w_{\la}^2)^-\in H^s(\Rn)$. Hence $w_{\la}\in H^s(\Rn)$.
\qed

\begin{lemma}
For every $\lambda\in\R$, 
\be\lab{29-9-1}
C\|w_{\la}\|_{p}^2\leq
\int_{\R^N} \big(f_\eps(u)-f_\eps(u_\la)\big)w_\la.
\ee
\end{lemma}

\proof
By applying Lemma \ref{l-test} and Sobolev inequality, we have $w_{\la}\in L^{2^*}(\Rn)$. Thus for $2<p< 2^*$, applying interpolation, $w_{\la}\in L^{p}(\Rn)$ and
$$\|w_{\la}\|_{p}\leq \|w_{\la}\|_{2}\|w_{\la}\|_{2^*}\leq C\|u\|_{L^2(\Rn)}\|w_{\la}\|_{\dot{H}^s(\Rn)}\leq C'\|w_{\la}\|_{\dot{H}^s(\Rn)}.$$
Similarly if $2^*<p<q$, then again by interpolation we have
$$\|w_{\la}\|_{p}\leq \|w_{\la}\|_{q}\|w_{\la}\|_{2^*}\leq C\|u\|_{L^q(\Rn)}\|w_{\la}\|_{\dot{H}^s(\Rn)}\leq C'\|w_{\la}\|_{\dot{H}^s(\Rn)}.$$
For $p=2^*$, clearly $\|w_{\la}\|_{2^*}\leq C'\|w_{\la}\|_{\dot{H}^s(\Rn)}.$

Moreover following the calculations in \cite[(3.28) and (3.29)]{DMPS}, we conclude that
\begin{equation}
\langle u-u_\la,w_\la\rangle_{\dot{H}^s(\R^N)}\ge\|w_{\la}\|_{\HH^s(\R^N)}^2.
\end{equation}
Hence \be\lab{24-10-2}\langle u-u_\la,w_\la\rangle_{\dot{H}^s(\R^N)}\ge C\|w_{\la}\|_{p}^2.\ee
Since $u$ is a weak solution of \eqref{P-eps}, it is easy to check that $u_{\la}$ is also a weak solution of \eqref{P-eps}. 
Now taking $w_{\la}$ as the test function for both the equations (satisfied by $u$ and $u_{\la}$) and subtracting one from the other, we get
\be\lab{24-10-3}\langle u-u_\la,w_\la\rangle_{\dot{H}^s(\R^N)}=\int_{\Rn}(f_{\eps}(u)-f_{\eps}(u_{\la}))w_{\la} dx.\ee
Combining \eqref{24-10-2} and \eqref{24-10-3}, we obtain \eqref{29-9-1}.
\qed

\begin{proof}[{\bf Proof of Theorem \ref{t:rad}}]
We split the proof into several steps.

\noindent
{\bf Step 1:}  {\em $\la_0:=\sup\{\la\in\R : u\leq u_\la$ in $\Sigma_\la\}$ is finite.}
	
For any $\la \in \R$, we  note that for any $t>1$,
$$0\leq\f{u^t-u_\la^t}{u-u_\la}\leq t\max\{u_{\la}^{t-1}, u^{t-1}\}.$$ Thus,
	
		\begin{eqnarray}
	\f{f_\eps(u)-f_\eps(u_\la)}{u-u_\la} &=& \f{u^{p-1}-u_\la^{p-1}}{u-u_\la}-\f{u^{q-1}-u_\la^{q-1}}{u-u_\la}-\eps\no\\
	&\leq& \f{u^{p-1}-u_\la^{p-1}}{u-u_\la}\leq (p-1) \max\{u_{\la}^{p-2}, u^{p-2}\}\label{e-p1}.
	\end{eqnarray}
We also observe that in $\Si_{\la}\cap(\text{supp}\, w_{\la})$, we have $u>u_{\la}$ and in 
 $\Si_{\la}^c\cap(\text{supp}\, w_{\la})$, we have $u<u_{\la}$. Therefore,
 \begin{equation}\lab{24-10-1}
 	\f{f_\eps(u)-f_\eps(u_\la)}{u-u_\la}\leq
	\begin{cases}
	 (p-1)u^{p-2}& \quad\text{in}\quad \Si_{\la}\cap(\text{supp}\, w_{\la})\\
	
	(p-1)u_{\la}^{p-2} &\quad\text{in}\quad \Si_{\la}^c\cap(\text{supp}\, w_{\la}).
\end{cases}
\end{equation}
Denote 
$$S_{\la}:=\Si_{\la}\cap(\text{supp}\, w_{\la}) \quad\text{and}\quad D_{\la}:= \Si_{\la}^c\cap(\text{supp}\, w_{\la}).$$	
Using \eqref{29-9-1} and \eqref{24-10-1}, similarly to the argument in \cite[(3.34)]{DMPS} we obtain
	\bea\lab{29-9-3}
	C\bigg(\int_{\Rn}|w_{\la}|^{p}\bigg)^{\f{2}{p}}&\leq &\int_{\Rn} \big(f_\eps(u)-f_\eps(u_\la)\big)w_{\la} \,dx\no\\
	&=& \int_{S_{\la}}\f{f_\eps(u)-f_\eps(u_\la)}{u-u_\la}(u-u_\la)w_{\la}   +  \int_{D_{\la}}\f{f_\eps(u)-f_\eps(u_\la)}{u-u_\la}(u-u_\la)w_{\la}\no \\
		&\leq&C_1\int_{S_\la} u^{p-2}w_{\la}^2+ C_1\int_{D_\la}u_\la^{p-2}w_{\la}^2 \no\\
			&\leq& C_1 \bigg(\int_{S_\la}u^p dx\bigg)^{\f{p-2}{p}}\bigg(\int_{S_\la}|w_{\la}|^{p}\bigg)^{\f{2}{p}}+ C_1 \bigg(\int_{D_\la}u_\la^p dx\bigg)^{\f{p-2}{p}}\bigg(\int_{D_\la}|w_{\la}|^{p}\bigg)^{\f{2}{p}} \no\\
			&\leq& C_2\bigg(\int_{S_\la}u^p dx\bigg)^{\f{p-2}{p}}\bigg(\int_{\Rn}|w_{\la}|^{p}\bigg)^{\f{2}{p}}.
	\eea
Moreover as $\la$ is big negative, we also observe that
$$\displaystyle\int_{S_\la}u^p(x) dx\leq \int_{\Sigma_\la}u^{p}(x)dx\leq\int_{(B_{|\la|})^c}u(y)^{p}dy.$$ 
Clearly, if $\la\to -\infty$, then the RHS of above expression converges to $0$.
Therefore, we can choose $R > 0$ big enough such that for all $\la<-R$,
$$C_2\bigg(\int_{S_\la}u^p dx\bigg)^{\f{p-2}{p}}<\frac{C}{2}.$$ Hence from \eqref{29-9-3}, we can conclude that 
$$w_{\la}\equiv 0 \,\, \text{for}\,\, \la<0\,\, \text{with}\,\, |\la| \,\, \text{sufficiently large}.$$
This implies $u \leq u_{\la}$ in $\Sigma_\la$ for all $\la <-R$, concluding that $\la_0 \geq -R.$ 
On the other hand, since $u(x)\to 0$ as $|x|\to\infty$, then there exists ${\la}_1$
such that $u(x)>u_{{\la}_1}(x)$ for some $x\in \Si_{{\la}_1}$. Hence $\la_0$ is finite.
	
		\vspace{2mm}

\noindent
{\bf Step 2}:  {\em $u \equiv u_{\la_0}$ in $\Sigma_{\la_0}$.}
	
	\vspace{2mm}
	
	We prove this step by the method of contradiction, that is, we suppose,
	$u \not\equiv u_{\la_0}$ and $u \leq u_{\la_0}$ in $\Sigma_{\la_0}$. First, we assume that there exists $x_0 \in \Sigma_{\la_0}$ such that
	$u_{\la_0}(x_0)=u(x_0),$ then we have,
	\begin{equation}\label{+++}
	(-\De)^s u_{\la_0}(x_0)-(-\De)^s u(x_0)=f_\eps(u_{\la_0}(x_0))-f_\eps(u(x_0))=0.
	\end{equation}
	On the other hand, by a direct computation (see \cite[Proof of Theorem 1.2, Step 2]{FW}) it can be easily checked that
	\be
	(-\De)^s u_{\la_0}(x_0)-(-\De)^s u(x_0) <0,
	\ee
	which contradicts \eqref{+++}. Consequently, $u<u_{\la_0}$ in $\Sigma_{\la_0}.$
	
	\vspace{2mm}
	
{\bf Claim:}	$u \leq u_\la$ in $\Sigma_\la$
	still holds even when $\la_0 < \la < \la_0 +\si$, where $\si>0$ is small. 
	
Following almost the same arguments as in \cite[Proof of Theorem 1.2, Step 2]{FW}, the claim can be proved. Below we briefly sketch the proof for the convenience of the readers. Convinced readers may skip this step and directly move to Step 3.	

 Let $\la \in (\la_0,\la_0+\si)$, where $\si\in(0, 1)$ will be chosen later.
	Let $P=(\la,0)$ and $B_R(P)$ be the ball centerd at $P$ with radius $R>1$ to be chosen later. Define $\tilde{B}=\Si_\la \cap B_R(P)$. 
We repeat now the argument above  using $w_{\la}$ as test function in the same fashion as we did before in step 1 and get again
	\begin{align}\label{step2_1}
	C\bigg (\int_{\Rn}|w_{\la}|^{p}dx\bigg)^{\f{2}{p}} \leq \bigg(\int_{S_\la}u^p dx\bigg)^{\f{p-2}{p}}\bigg(\int_{\Rn}|w_{\la}|^{p}\bigg)^{\f{2}{p}}.
		\end{align}
We estimate the integral on the right. 
\begin{eqnarray}\lab{25-10-1}
\int_{S_\la}u^p dx&=& \int_{S_{\la}\cap\tilde B}u^p\, dx+ \int_{S_\la\setminus\tilde B}u^p \,dx\no\\
&\leq&\int_{\tilde B\cap\,(\text{supp}\, w_{\la})}u^p \, dx+\int_{\Si_\la\setminus\tilde B}u^p \,dx\no\\
&\leq& C'|\tilde B\cap\,(\text{supp}\, w_{\la})|+\int_{\Si_\la\setminus\tilde B}u^p \,dx,
\end{eqnarray}	
where $C'=\|u\|^p_{\infty}$. On the other hand, for the integral over $\Sigma_\la \setminus \tilde{B},$ we assume $R$ and $R_0$ are such that
	$\Sigma_\la \setminus \tilde{B} \subset (B_R(P))^c \subset (B_{R_0}(0))^c$. Therefore,


$$\int_{\Si_\la\setminus\tilde B}u^p \,dx\leq \int_{(B_{R_0}(0))^c}u^p dx.$$

	We choose $R_0$ such that $R_0>|\la_0|$ and $\displaystyle\int_{(B_{R_0}(0))^c}u^p\ dx<\frac{C}{4},$ where $C$ is as in \eqref{step2_1}. 
	Then we choose $R$ such that $R>R_0+|\la_0|+1$. Since, $\la\in (\la_0, \la_0+\sigma)$ and $\sigma\in(0,1)$, clearly $|\la|<|\la_0|+1$ and thus $R>R_0+|\la|$. As $P=(\la, 0)$, the choice of $R$ implies $B_{R_0}(0)\subset B_R(P)$. Hence
	$\Sigma_\la \setminus \tilde{B} \subset (B_R(P))^c\subset (B_{R_0}(0))^c$. Then choose $\si\in(0,1)$ so that $|\tilde{B}\cap (\mbox{supp}\, w_\la)|<\frac{C}{4C'}.$
	With this choice of the parameters, from \eqref{25-10-1} and \eqref{step2_1}, it
	follows that  $\f{C}{2}\bigg (\displaystyle\int_{\Rn}|w_{\la}|^{p}dx\bigg)^{\f{2}{p}} <0$. Hene $w_{\la} = 0$ in $\Rn$, which implies $u\leq u_{\la}$ in $\Si_{\la}$, where $\la_0 < \la < \la_0 +\si$. Thus the claim follows.
	
Clearly the above claim contradicts the definition of $\la_0$. Hence, Step 2 is proved. 
	

\vspace{2mm}
	
\noindent
{\bf Step 3:} 
Using translation, we may say that $\la_0=0.$ Repeating the argument from the other side, we find that $u$ is symmetric about $x_1-$axis. Using the same argument in any arbitrary direction, we finally conclude that $u$ is radially symmetric. Proceeding as in \cite[Theorem 1.1]{FW}, it can be easily checked that $u$ is strictly decreasing. Hence the theorem follows.
\end{proof}


\section{Berestycki--Lions characterization of the ground states}

\subsection{Existence of the constrained minimizers}

Denote
$$\eps_*:=\inf\{\eps>0\,:\,F_\eps(\zeta)\le 0\quad\forall\zeta>0\}.$$
It is clear that $\eps_*>0$ and is finite. Using Poho\v zaev identity of Lemma \ref{l:Phz}, we immediately conclude that \eqref{P-eps}
does not have solutions for $\eps\ge \eps_*$. Indeed, if $u\in H^s(\R^N)$ is a weak solution of \eqref{P-eps} with $\eps\ge\eps_*$,
then $\|u\|_{\dot{H}^s}^2\le 0$.

\begin{proposition}\label{t:ext}
Let $N>2s$ and $q>p>2$. Then for every $\eps\in(0,\eps_*)$ the problem
\begin{align}\label{S-eps}
S_\eps:=\inf\bigg\{\displaystyle\|w\|_{\HH^s(\R^N)}^2  : \, w \in H^s(\Rn),\quad 2^*\int_{\Rn}F_\eps(w)dx=1\bigg\}.
\end{align}
admits a minimizer $w_\eps$. Moreover, $w_\eps$ is a radially symmetric and decreasing H\"older continuous function of $|x|$ which is a classical solution of the Euler-Lagrange equation
\begin{align}\label{lag-mult}
(-\De)^s w_\eps=S_\eps f_\eps(w_\eps) \quad\mbox{in}\quad \Rn.
\end{align}
In addition, $0<w_\eps(x)\le 1$ for all $x\in\R^N$ and
\begin{equation}
w_\eps(x)= C_\eps|x|^{-(N+2s)}+o(|x|^{-(N+2s)})\quad\text{as }\,|x|\to\infty,
\end{equation}
where $C_\eps>0$ depends on $N, s, p, q$ and $\eps$.
\end{proposition}

\vspace{3mm}
\begin{remark}
Relating to minimization problem \eqref{S-eps}, we define an equivalent scaling--invariant quotient
\begin{equation}\label{M-scale}
\mathcal{S}_\eps(w):=\f{\|w\|_{\HH^s(\R^N)}^2}{\bigg(2^*\displaystyle\int_{\Rn}F_\eps(w)dx\bigg)^{\f{N-2s}{N}}},\qquad w \in \mathcal{M}_\eps,
\end{equation}
where
\begin{align}\label{M-eps}
\mathcal{M}_\eps:=\bigg\{0 \leq u \in H^s(\Rn):\int_{\Rn}F_\eps(w)dx>0\bigg\}.
\end{align}
By setting $w_\la(x):=w(\f{x}{\la}),$ we see that
$\mathcal{S}_\eps(w_\la)=\mathcal{S}_\eps(w)$ for all $\la>0,$ that is, $\mathcal{S}_\eps(\cdot)$ is invariant under dilations, which implies
\begin{align}\label{inf-M-eps}
S_\eps=\inf_{w \in \mathcal{M}_\eps}\mathcal{S}_\eps(w).
\end{align}
Furthermore, note that as $\mathcal{M}_{\eps_2} \subset \mathcal{M}_{\eps_1}$ for $\eps_2 > \eps_1>0,$ \eqref{inf-M-eps} shows that
$\mathcal{S}_\eps$ is a nondecreasing function of $\eps\in(0,\eps_*)$.
In our analysis of asymptotic profiles of minimizers $w_\eps$ in Section \ref{s-profile} instead of $S_\eps$ we will be using the equivalent minimization problem
\eqref{inf-M-eps}.
\end{remark}

\begin{proof}[{\bf Proof of Proposition \ref{t:ext}}]
The proof follows closely the arguments of  Berestycki and Lions \cite[Theorem 2]{BL},
so we only highlight the main steps.
\vspace{2mm}

Similarly to \cite[p.323]{BL}, we introduce the truncated nonlinearity
\begin{align}\label{F-eps}
\bar f_\eps(u):=\begin{cases}
0, \,\,\,\,\,\ u<0,\\
u^{p-1}-u^{q-1}-\eps u, \quad u \in [0,1],\\
-\eps, \,\,\ u>1,
\end{cases}
\end{align}
and denote
$$\bar F_\eps(u):=\int_0^u \bar f_\eps(s)\,ds.$$
Taking into account Lemma \ref{l:max}, there is no loss of generality in replacing $f_\eps$ by $\bar f_\eps$,
so in the rest of the proof we adopt the convention that {\em $f_\eps$ had been replaced by $\bar f_\eps$, however we keep the notations $f_\eps$ and $F_\eps$.}
\vspace{2mm}

Following exactly the same arguments as in \cite[p.325]{BL}, we can show
that the set
$$M_\eps:=\bigg\{w \in H^s(\Rn): \, 2^*\displaystyle\int_{\Rn} F_\eps(w)dx=1\bigg\}$$
is non-empty if and only if $\eps\in(0,\eps_*)$.
\vspace{2mm}


%

\vspace{2mm}

Let $\{w_n\} \subset H^s(\Rn)$ be a minimizing sequence of $S_{\eps}$ that is,
$$2^*\displaystyle\int_{\Rn}F_\eps(w_n)=1, \qquad \lim_{n \to \infty} \|w_n\|_{\HH^s(\R^N)}^2=S_\eps.$$
Using symmetric rearrangement technique in $H^s(\R^N)$ (see  \cite{FS-1}),
without loss of generality we can assume that  $w_n$ is radially symmetric, nonnegative and decreasing.
\vspace{2mm}

{\bf Claim 1:} $\{w_n\}$ is bounded in $H^s(\Rn)\cap L^q(\Rn)$.

\vspace{2mm}

To prove the claim, we note that $\|w_n\|_{\HH^s(\R^N)}$ is uniformly bounded. Therefore by Sobolev inequality, $w_n$ is uniformly bounded in $L^{2^*}(\Rn)$.  If $p=2^*$, then we immediately conclude that $w_n$ is uniformly bounded in $L^p(\Rn)$. Moreover, $2^*\displaystyle\int_{\Rn}F_{\eps}(w_n)dx=1$  implies
 \be\lab{4-17-1'}
 1+\f{2^*}{2}\eps\|w_n\|_2^2 +\f{2^*}{q}\|w_n\|_q^q=\f{2^*}{p}\|w_n\|_p^p.
 \ee As a result,
 \be\lab{w-n-1}
(i)\quad\|w_n\|_{2}<C\|w_n\|_p^\f{p}{2} \quad\text{and}\quad (ii)\quad\|w_n\|_{q}<C\|w_n\|_p^\f{p}{q}.
 \ee
Thus if $2<p<2^*$, we get $$\|w_n\|_p\leq \|w_n\|^{\theta}_2\|w_n\|_{2^*}^{1-\theta}\leq C\|w_n\|^{\theta}_2\leq C\|w_n\|_p^\f{p\theta}{2},$$
where $\f{1}{p}=\f{\theta}{2}+\f{1-\theta}{2^*}$. This in turn implies, $\|w_n\|_p\leq C$. Similarly, if $p>2^*$ then as $2^*<p<q$, again using interpolation and \eqref{w-n-1}(ii), it can be shown that  $\|w_n\|_p\leq C$.
So for any $p$ with $2<p<q$, we have $w_n$ is uniformly bounded in $L^p(\Rn)$ and consequently by \eqref{4-17-1'}, $w_n$ is uniformly bounded in $L^2(\Rn)$ and $L^q(\Rn)$ and hence in $H^s(\Rn)\cap L^q(\Rn)$. This proves the claim.

\vspace{2mm}

Therefore, up to a subsequence, $w_n\deb w$ in $H^s(\Rn)$ and $w_n(x)\to w(x)$ a.e.. Clearly, $w$ is radially symmetric, nonnegative, decreasing in $|x|$. By Fatou's lemma combined with Claim 1, it follows $w\in L^q(\Rn)$. Observe that from \cite[Lemma 6.1]{BM}, it follows  $H^s_{rad}(\Rn)\hookrightarrow L^t(\Rn)$ compactly for $2\leq t<2^*$. Consequently, $w_n\to w$ in $L^2(\Rn)$. Therefore, by interpolation inequality we have
$$\|w_n-w\|_p\leq \|w_n-w\|_2^{\theta}\|w_n-w\|_q^{1-\theta}\leq C\|w_n-w\|_2^{\theta}\to 0.$$
Hence by Fatou's lemma, $2^*\displaystyle\int_{\Rn}F_{\eps}(w)dx\geq 2^*\displaystyle\int_{\Rn}F_{\eps}(w_n)dx=1$.

\vspace{2mm}

{\bf Claim 2}: $2^*\displaystyle\int_{\Rn}F_{\eps}(w)dx=1$.

\vspace{2mm}

Suppose not, i.e., $2^*\displaystyle\int_{\Rn}F_{\eps}(w)dx>1$. Then define, $w_{\de}(x)=w(\f{x}{\de})$, which implies $$2^*\displaystyle\int_{\Rn}F_{\eps}(w_{\de})dx=2^*\de^N\displaystyle\int_{\Rn}F_{\eps}(w)dx=1, \quad\text{for some}\quad 0<\de<1.$$
Since norm is weakly lower semicontinuous we have 
\be\lab{4-17-5}
\|w\|_{\dot{H}^s(\Rn)}^2\leq \lim\inf_{n\to\infty}\|w_n\|_{\dot{H}^s(\Rn)}^2=S_{\eps}.
\ee
This implies $\|w_\de\|_{\dot{H}^s(\Rn)}^2=\de^{N-2s}\|w\|_{\dot{H}^s(\Rn)}^2\leq \de^{N-2s}S_{\eps}$. On the other hand, from the very definition of $S_{\eps}$, it follows $S_{\eps}\leq \|w_\de\|_{\dot{H}^s(\Rn)}^2$. Therefore, $S_{\eps}=0$ and thus \eqref{4-17-5} implies $\|w_\de\|_{\dot{H}^s(\Rn)}=0$, i.e., $w=0$. This contradicts the fact that $2^*\displaystyle\int_{\Rn}F_{\eps}(w_{\de})dx>0$. Hence the claim follows.


\vspace{2mm}

Therefore by definition of $S_{\eps}$, we have $S_{\eps}\leq \|w\|_{\HH^s(\R^N)}^2$. Combining this with the weak lower semicontinuity of $\|\cdot\|_{\dot{H}^s}$--norm, we conclude $w$ is a minimizer  of  $S_{\eps}$. Note that, throughout this process we have kept $\eps\in (0,\eps_*)$ fixed. So, corresponding to each fixed $\eps$ we obtain a minimizer of $S_{\eps}$, which we denote as $w_{\eps}$.
We observe that $w_{\eps}$ is a radially symmetric and decreasing function of $|x|$ (since $w_{\eps}$ is the limit of $w_n$).

Moreover, by Lagrange multiplier principle, there exists $\theta_\eps>0$ such that $w_\eps$ is a weak solution of
\begin{align}\label{lag-mult}
 (-\De)^s w_\eps=\theta_\eps f_\eps(w_\eps) \quad\mbox{in}\quad \Rn.
\end{align}
In particular, the minimizer $w_\eps$ satisfies Nehari's identity
\begin{align}\label{Nehari}
\displaystyle\|w_\eps\|_{\HH^s(\R^N)}^2 =\theta_\eps \int_{\Rn}f_\eps(w_\eps)w_\eps dx,
\end{align}
and, by Lemma \ref{l:Phz}, Poho\v zaev's identity
\begin{align}\label{poho}
\displaystyle\|w_\eps\|_{\HH^s(\R^N)}^2 =\theta_\eps 2^* \int_{\Rn}F_\eps(w_\eps)dx,
\end{align}
which in turn implies 
\begin{align}\label{theta}
\theta_\eps=S_\eps.
\end{align}
Positivity, regularity and asymptotic proprties of $w_\eps$ follow from the results in Section \ref{s-qual},
with suitable adjustments made to accommodate the Lagrange multiplier.
\end{proof}

\subsection{From minimizers to ground states -- Proof of Theorem \ref{t:BL}}
We define
\begin{align}\label{u-eps}
 u_\eps(x):=w_\eps\bigg(\f{x}{S_{\eps}^{1/2s}}\bigg),
\end{align}
where $w_\eps$ is found as in Proposition \ref{t:ext}. A direct calculation using \eqref{theta}, shows that \eqref{u-eps} is the radial solution of \eqref{P-eps},
which satisfies all the properties declared in Theorem \ref{t:BL}.
We only need to show that $u_\eps$ is a ground state, i.e. it has minimal energy amongst all nontrivial weak solutions of \eqref{P-eps}.

Indeed, we compute
$$\|u_{\eps}\|^2_{\dot{H}^s(\Rn)}=S_{\eps}^\f{N-2s}{2s}\|w_{\eps}\|^2_{\dot{H}^s(\Rn)}.$$
Therefore
\Bea
E_\eps(u_{\eps})&=&\f{1}{2}\displaystyle\|u_\eps\|_{\HH^s(\R^N)}^2 -\displaystyle\int_{\Rn}F_{\eps}(u_{\eps})dx\\
&=&\f{1}{2}S_{\eps}^\f{N-2s}{2s}\|w_\eps\|_{\HH^s(\R^N)}^2 - S_{\eps}^\f{N}{2s}\displaystyle\int_{\Rn}F_{\eps}(w_{\eps})dx\\
&=&\f{1}{2}S_{\eps}^\f{N}{2s}-\f{1}{2^*}S_{\eps}^\f{N}{2s}=\f{s}{N}S_{\eps}^\f{N}{2s}.
\Eea
On the other hand, if $v$ is any weak solution of \eqref{P-eps}, then using Poho\v zaev identity (Lemma \ref{l:Phz}), it follows that $E_\eps(v)=\f{s}{N}\|v\|_{\HH^s(\R^N)}^2$. Furthermore, combining Poho\v zaev identity along with the definition of $S_\eps$, we have
$$S_{\eps}\leq \f{\|v\|_{\HH^s(\R^N)}^2}{\bigg(2^*\displaystyle\int_{\Rn}F_\eps(v)dx\bigg)^{\f{2}{2^*}}}=\bigg(\|v\|_{\HH^s(\R^N)}^2\bigg)^\f{2s}{N}.$$ Therefore, $E_\eps(v)\geq \f{s}{N}S_{\eps}^\f{N}{2s}$. This implies that $u_\eps$ is a ground state, that is, a nontrivial solution with the least energy.
\qed

\section{Asymptotic profiles in the critical case $p=2^*$}\label{s-profile}

Throughout this section we assume that $p=2^*$  and $N>4s$.

\subsection{Critical Emden-Fowler equation}

Let
\begin{align}\label{S*}
S_*:=\inf{\|w\|_\Hs^2 :\, w \in \dot{H}^s(\Rn),\int_{\Rn}|w|^pdx=1\bigg\}}
\end{align}
be the optimal constant in the fractional Sobolev inequality
$$\displaystyle \|w\|_\Hs^2 \geq S_*\bigg(\int_{\Rn}|w|^pdx\bigg)^{2/p}, \quad\mbox{for all}\quad w \in \dot{H}^s(\Rn).$$
We note that $S_*$ is achieved by translations of the rescaled family
\be\lab{W-la}
W_\la(x):=U_\la\big(S_*^{\f{1}{2s}}x\big),
\ee
where $U_\la$ are the ground states of the critical Emden-Fowler equation \eqref{R-*}.
It is not difficult to check that, $\|W_\la\|_p=1, \ \|W_\la\|^2_{\dot{H}^s(\Rn)}=S_*.$

A straightforward computation leads to the explicit expression $\|U_\la\|^2_{\dot{H}^s(\Rn)}=\|U_\la\|_p^p=S_*^{\f{N}{2s}}.$
We observe that the family of minimizers $W_\la$ solves the Euler-Lagrange equation
$$(-\De)^s W=S_* W^{p-1} \quad\mbox{in}\quad \Rn.$$

\subsection{Variational estimates of $S_\eps$}
Let us consider the dilation invariant Sobolev quotient
$$\mathcal{S}_*(w):=\f{\displaystyle\|w\|_\Hs^2
 }{\bigg(\displaystyle\int_{\Rn}|w|^pdx\bigg)^{\f{(N-2s)}{N}}}, \quad w \in \dot{H}^s(\Rn), \ w \neq 0,$$
Thus $$S_*=\inf_{0 \neq w \in \dot{H}^s(\Rn)}\mathcal{S}_*(w).$$
Define, \be\lab{4-12-1}\sigma_\eps:=S_\eps-S_*.\ee
In order to control $\sigma_\eps$ in terms of $\eps$, we will use minimizer of Sobolev inequality i.e., $W_\la$, as a family of test functions for $\mathcal{S}_\eps$. Note that for $N>4s$,  $W_{\la}\in L^2(\Rn)$ and
$$\|W_\la\|_t^t=\la^{N-\f{t(N-2s)}{2}}\|W_1\|_t^t.$$ That is, in particular, 
$$\|W_\la\|_2^2=\la^{2s}\|W_1\|_2^2.$$

\begin{lemma}\label{lem-1}
Let $\sigma_{\eps}$ is as defined in \eqref{4-12-1}. Then
\begin{equation}\label{bnd1}
  0< \sigma_\eps \lesssim  \eps^{\f{q-p}{q-2}}.
 \end{equation}
In particular, $\sigma_\eps \to 0$ as $\eps \to 0.$
\end{lemma}
\begin{proof}
 As $S_* \leq \mathcal{S}_*(w_\eps)<\mathcal{S}_\eps(w_\eps)=S_\eps$,  we have $\sigma_\eps>0$. Next we will establish an upper estimate for $\sigma_{\eps}$.
 Using $W_\la$ as a family of test functions,
 we obtain $W_\la \in \mathcal{M}_\eps$ for sufficiently small $\eps$. Consider $\la$ sufficiently large, so that
 $$F_\eps(W_\la)=\f{1}{p}W_\la^p-\f{1}{q}W_\la^q-\f{\eps}{2}W_\la^2.$$ Thus we have,
 \Bea
 \int_{\Rn}F_\eps(W_\la)&=& \f{1}{p}\|W_\la\|_p^p -\f{1}{q}\|W_\la\|_q^q-\f{\eps}{2}\|W_\la\|_2^2\\
 &=& \f{1}{p} -\f{1}{q}\la^{N-\f{q(N-2s)}{2}}\|W_1\|_q^q-\f{\eps}{2}\la^{2s}\|W_1\|_2^2.\\
 \Eea
Therefore,
$$\mathcal{S}_\eps(W_\la) \leq \f{S_*}{\big(1-\ba_2 \eps \la^{2s}-\ba_q \la^{N-\f{q(N-2s)}{2}}\big)^{\f{N-2s}{N}}},$$
where $\ba_2=\f{2^*}{2}\|W_1\|_2^2$ and $\ba_q=\f{2^*}{q}\|W_1\|_q^q$. Clearly,
\be\lab{4-13-1}
S_{\eps}\leq\min_{\la>0}\mathcal{S}_\eps(W_\la) \leq \min_{\la>0}\f{S_*}{\big(1-\ba_2 \eps \la^{2s}-\ba_q \la^{N-\f{q(N-2s)}{2}}\big)^{\f{N-2s}{N}}}.\ee
So, to minimize the function on RHS, we need to minimize the scalar function $\psi$ where $$\psi(\la):=\ba_2\eps\la^{2s}+\ba_q \la^{N-\f{q(N-2s)}{2}},$$ keeping $\psi<1$ .
Now, $$\psi'(\la)=2s\ba_2\eps\la^{2s-1}+\ba_q\bigg(N-\f{q(N-2s)}{2}\bigg)\la^{N-1-\f{(N-2s)q}{2}}.$$
So, $\psi'(\la)=0$ implies $$\la=\la_\eps=\bigg(\f{2s\ba_2 \eps}{\ba_q \big(\f{N-2s}{2}q-N\big)}\bigg)^{\f{-2}{(q-2)(N-2s)}}.$$
Again,
$$\psi''(\la)=2s(2s-1)\ba_2\eps\la^{2s-2}+\ba_q\bigg(\f{(N-2s)q}{2}-N\bigg)\bigg(\f{N-2s}{2}q-(N-1)\bigg)\la^{N-2-\f{(N-2s)q}{2}}.$$
Note that  if $s \geq \f{1}{2},$ then $\psi''(\la)>0$  for all $\la>0$. On the other hand, if $s < \f{1}{2}$ then it is not difficult to check that $\psi''(\la_{\eps})>0$, as the second term in the expression of $\psi''(\la)$ becomes the dominating one. Thus $$\min_{\la}\psi(\la)=\psi(\la_{\eps})\sim \eps\la_{\eps}^{2s}+\la_{\eps}^{N-\f{N-2s}{2}q}=\eps^\f{(N-2s)q-2N}{(N-2s)(q-2)}=\eps^\f{q-p}{q-2}.$$
Consequently, from \eqref{4-13-1} we have
 $$\mathcal{S}_\eps\leq \f{S_*}{\big(1-\psi(\la_\eps)\big)^{\f{N-2s}{N}}}=S_*(1+O(\psi(\la_\eps)))=S_*+O(\eps^{\f{q-p}{q-2}}).$$ Hence $\sigma_{\eps} \lesssim \eps^{\f{q-p}{q-2}}$.
\end{proof}

\subsection{Poho\v zaev estimates}
Nehari identity \eqref{Nehari} together with Poho\v zaev's identity \eqref{poho} lead to the following important relations.
\begin{lemma}\label{lem-2}
Let $q>p=2^*>2$ and $k :=\f{q(p-2)}{2(q-p)}>0$ and $w_{\eps}$ be a minimizer of $S_{\eps}$ (see \eqref{S-eps}). \\
Then $\|w_\eps\|_q^q=k\eps\|w_\eps\|_2^2$ and $\|w_\eps\|_p^p=1+(k+1)\eps\|w_\eps\|_2^2.$
\end{lemma}

\begin{proof}
 As $w_\eps$ is a minimizer of \eqref{S-eps}, identities \eqref{lag-mult}-\eqref{poho} yields us
 \Bea
\theta_\eps=\theta_\eps 2^* \int_{\Rn}F_\eps(w_\eps)dx= \|w_{\eps}\|^2_{\dot{H}^s(\Rn)}&=&\theta_\eps \int_{\Rn} f_\eps(w_\eps)w_\eps dx\\
 &=& \theta_\eps \big(\|w_\eps\|_p^p-\|w_\eps\|_q^q-\eps\|w_\eps\|_2^2\big),
 \Eea
which  implies
 \begin{equation}\label{i}
  1=\|w_\eps\|_p^p-\|w_\eps\|_q^q-\eps\|w_\eps\|_2^2.
 \end{equation}
On the other hand, $2^* \displaystyle\int_{\Rn}F_\eps(w_\eps)dx=1$ implies
 \begin{equation}\label{ii}
 \|w_\eps\|_p^p-\f{p}{q}\|w_\eps\|_q^q-\f{\eps p}{2}\|w_\eps\|_2^2=1.
 \end{equation}
 Eliminating $\|w_\eps\|_p^p$ from \eqref{i} and \eqref{ii} we get,
 $\|w_\eps\|_q^q=k\eps\|w_\eps\|_2^2.$ Similarly, eliminating $\|w_\eps\|_q^q$ we obtain,
 $\|w_\eps\|_p^p=1+(k+1)\eps\|w_\eps\|_2^2.$
\end{proof}

\begin{lemma}\label{lem-3}
 $\eps(k+1)\|w_\eps\|_2^2 \leq \f{N}{N-2s}S_*^{-1}\sigma_\eps(1+o(1)).$
\end{lemma}
\begin{proof}
Using Lemma \ref{lem-2} together with the fact that $w_\eps$ is a minimizer of \eqref{S-eps}, we obtain
$$S_* \leq \mathcal{S}_*(w_\eps)=\f{\displaystyle\|w\|_\Hs^2}{\bigg(\displaystyle\int_{\Rn}|w_\eps|^pdx\bigg)^\frac{2}{p}}=\f{S_\eps}{\big(1+(k+1)\eps\|w_\eps\|_2^2\big)^{2/p}}.$$
This implies,
$$S_*^{\f{N}{N-2s}}(k+1)\eps \|w_\eps\|_2^2 \leq S_\eps^{\f{N}{N-2s}}-s_*^{\f{N}{N-2s}}.$$
Since, $\sigma_\eps:=S_\eps-S_*$,  using Taylor's expansion we have,
$$S_\eps^{\f{N}{N-2s}}=S_*^{\f{N}{N-2s}}+\f{N}{N-2s}\sigma_\eps S_*^{\f{N}{N-2s}-1}+o(\sigma_\eps).$$
Thus,
\Bea
 S_*^{\f{N}{N-2s}}(k+1)\eps \|w_\eps\|_2^2 &\leq& \f{N}{N-2s}S_*^{\f{2s}{N-2s}}\sigma_\eps + o(\sigma_\eps)
 = \sigma_\eps \bigg[\f{N}{N-2s}S_*^{\f{2s}{N-2s}} + o(1)\bigg].
\Eea
Hence the lemma follows.
\end{proof}
\begin{corollary}\label{cor-4}
$\|w_\eps\|_p^p \to 1,\,\ \|w_\eps\|_q^q \to 0,\,\ \eps\|w_\eps\|_2^2 \to 0$ as $\eps \to 0.$
\end{corollary}
The proof of the corollary follows from Lemma \ref{lem-1}--Lemma \ref{lem-3} . It gives us the asymptotic behaviour of different norms
associated with the minimizer $w_\eps$ of \eqref{S-eps}.

\subsection{Optimal rescaling}
In the spirit of \cite{Lions}, we define the concentration function
$$Q_\eps(\la):=\int_{B_\la}|w_\eps|^pdx,$$
where $B_\la$ is the open ball in $\Rn$ of radius $\la$ centred at the origin. Clearly, $Q_\eps(.)$ is strictly increasing with
$\lim_{\la \to 0} Q_\eps(\la)=0$ and $\lim_{\la \to \infty} Q_\eps(\la)=\|w_\eps\|_p^p \to 1$ as $\eps \to 0$, in view of Corollary \ref{cor-4}. Let $W_\la$ be defined as in \eqref{W-la}. We note that, $\|W_1\|_{L^p(\Rn)}=1$ and $W_1>0$ implies $\|W_1\|_{L^p(B_1)}<1$.
Therefore,  there exists $\la=\la(\eps)>0$ such that
\begin{equation}\label{Q-*}
 Q_\eps(\la_\eps)=Q_*:=\int_{B_1}|W_1(x)|^pdx<1.
\end{equation}
Similarly, we define the function
$$Q_0(\la):=\int_{B_{\f{1}{\la}}}|W_1(x)|^pdx=\int_{B_1}|W_\la(x)|^pdx$$
which is strictly decreasing with $\lim_{\la \to 0}Q_0(\la)=1$
and $\lim_{\la \to \infty}Q_0(\la)=0.$ Thus, there is a unique solution to the equation
$$Q_0(\la)=Q_*=\displaystyle\int_{B_1}|W_1(x)|^pdx.$$
Therefore, $\la=1.$
Using the value of $\la_\eps$ implicitly determined by \eqref{Q-*}, we define the rescaled family
\begin{equation}\label{v-eps}
 v_\eps(x):=\la_\eps^{\f{(N-2s)}{2}}w_\eps(\la_\eps x).
\end{equation}
We observe that
\begin{eqnarray}\label{v-1}
 \|v_\eps\|_p&=&\|w_\eps\|_p=1 +o(1), \no\\
\|v_\eps\|_{\Hs}^2&=&\|w_\eps\|_{\Hs}^2=S_* +o(1), \quad\mbox{as}\quad \eps \to 0,
\end{eqnarray}
that is, $(v_\eps)$ is a minimizing sequence for $S_*.$ Also, we note that
\be\lab{4-14-1}
\int_{B_1}|v_\eps(x)|^pdx=\int_{B_{\la_\eps}}|w_\eps(x)|^pdx=Q_\eps(\la_\eps)=Q_*.\ee
\begin{lemma}\label{lem-5}
 $\|v_\eps-W_1\|_{\dot{H}^s(\Rn)} \to 0$ and $\|v_\eps-W_1\|_p \to 0$ as $\eps \to 0.$
\end{lemma}
\begin{proof}
 For any sequence $\eps_n \to 0$, there exists a subsequence $(\eps_n')$ such that
 $(v_{\eps_n'})$ converges weakly in
 $\dot{H}^s(\Rn)$ to some radial function $w_0 \in \dot{H}^s(\Rn)$.
 Let $\tilde{v}_{\eps_n'}=\f{v_{\eps_n'}}{|v_{\eps_n'}|_p}.$
 Note as $\|v_{\eps_n'}\|_p=1+o(1)$ as $n \to \infty,$ we have $\tilde{v}_{\eps_n'} \rightharpoonup w_0\not=0$ in $\dot{H}^s(\Rn)$ (see \cite[Lemma 4.1]{PP}).
 Also, $\|\tilde{v}_{\eps_n'}\|_p=1$ for all $n \geq 1$ and $||\tilde{v}_{\eps_n'}||_{\dot{H}^s(\Rn)}^2=S_*+o(1)$ as $n \to \infty.$
 Hence, proceeding as in the proof of  \cite[Theorem 1.3]{PP}, up to a subsequence, we obtain
 \begin{equation}\label{a1}
\|w_0\|_{\dot{H}^s}^2+\limsup_{n \to \infty} \|\tilde{v}_{\eps_n'}-w_0\|_{\dot{H}^s}^2=\limsup_{n \to \infty} \|\tilde{v}_{\eps_n'}\|_{\dot{H}^s}^2=S_*,
 \end{equation}
and
\begin{equation}\label{a2}
 \|w_0\|^p_p+\lim_{n \to \infty} \|\tilde{v}_{\eps_n'}-w_0\|^p_p=\lim_{n \to \infty} \|\tilde{v}_{\eps_n'}\|^p_p=1.
\end{equation}
Combining \eqref{a1} and \eqref{a2} along with Sobolev inequality and an elementary inequality for positive numbers $(a+b)^{\f{2}{p}} \leq a^{\f{2}{p}}+b^{\f{2}{p}}$, as $0<\f{2}{p}<1$, we have
\Bea
S_*&=& \|w_0\|_{\dot{H}^s}^2+\limsup_{n \to \infty} \|\tilde{v}_{\eps_n'}-w_0\|_{\dot{H}^s}^2\\
&\geq& S_* \big(\|w_0\|_p^2+\limsup\|\tilde{v}_{\eps_n'}-w_0\|_p^2 \big)\\
&=& S_* \big(\|w_0\|_p^2+\lim\|\tilde{v}_{\eps_n'}-w_0\|_p^2 \big)\\
&=& S_* \big((\|w_0\|_p^p)^{\f{2}{p}}+(\lim|\tilde{v}_{\eps_n'}-w_0|_p^p)^{\f{2}{p}} \big)\\
&\geq& S_* \big(\|w_0\|_p^p+\lim\|\tilde{v}_{\eps_n'}-w_0\|_p^p \big)^{\f{2}{p}}\\
&=& S_*,
\Eea
Therefore, all the inequalities in the above expressions are equality.
Define, $A:=S_*^{-1}\|w_0\|_{\dot{H}^s}^2$,  $B:=S_*^{-1}\limsup_{n \to \infty} \|\tilde{v}_{\eps_n'}-w_0\|_{\dot{H}^s}^2.$
Note that $A+B=1.$ So, $0 \leq A, B \leq 1.$ As $w_0 \nequiv 0,$ we have $A \neq 0.$ Thus, $0 \leq B <1.$ We claim that $B=0.$
Suppose not, then, $0<B<1$ which implies  $B^{\f{p}{2}}< B$. Similarly, $0 \leq A \leq 1,$ implies $A^{\f{p}{2}}\leq A.$
Consequently, by Sobolev inequality and \eqref{a2} we have
\Bea
1&=& A+B\\
&>& A^{\f{p}{2}}+ B^{\f{p}{2}}\\
&=& S_*^{\f{-p}{2}}\big(\|w_0\|_{\dot{H}^s}^p+\limsup_{n \to \infty} \|\tilde{v}_{\eps_n'}-w_0\|_{\dot{H}^s}^p \big)\\
&\geq& S_*^{\f{-p}{2}}\big(S_*^{p/2}\|w_0\|_p^p+S_*^{p/2}\lim\|\tilde{v}_{\eps_n'}-w_0\|_p^p \big)\\
&=& 1,
\Eea
 which is a contradiction. Thus, $B=0$, that is, $\limsup_{n \to \infty}\|\tilde{v}_{\eps_n'}-w_0\|_{\dot{H}^s}^2=0$, which implies $\tilde{v}_{\eps_n'} \to w_0$
 in $\dot{H}^s(\Rn)$ and $L^p(\Rn)$. Hence, $v_{\eps_n'}\to w_0$  in $\dot{H}^s(\Rn)\cap L^p(\Rn)$, as  $\|v_{\eps_n'}\|_p \to 1$ as $n \to \infty$. As a consequence,
$\|w\|_p=1$. Moreover, $B=0$ implies $A=1$. Therefore, $w_0$ is a radial minimizer
 of \eqref{S*}, that is, $w_0 \in \{W_\la\}_{\la>0}$. Furthermore, from \eqref{4-14-1} it follows $\displaystyle\int_{B_1}|w_0(x)|^pdx=Q_*$. Therefore, we
 conclude that $w_0=W_1.$ Finally, by uniqueness of the limit the full sequence $(v_{\eps_n})$ converges to $W_1$ strongly in
 $\dot{H}^s(\Rn)$ and $L^p(\Rn).$

\end{proof}

\subsection{Rescaled equation estimates}
Doing a straight-forward computation using \eqref{lag-mult} and \eqref{theta}, it can be checked that the rescaled minimizer $v_\eps$ defined in \eqref{v-eps} solves the equation
\begin{equation}\lab{4-20-2}
 (-\De)^s v_\eps(x)+S_\eps \eps \la_\eps^{2s}v_\eps(x)=S_\eps \big(v_\eps^{p-1}(x)-\la_\eps^{\f{-2s(q-p)}{p-2}}v_\eps^{q-1}(x)\big).
\end{equation}
From the definition of $v_\eps$ we obtain,
$$\|v_\eps\|_q^q=\la_\eps^{\f{2s(q-p)}{(p-2)}}\|w_\eps\|_q^q,\,\ \|v_\eps\|_2^2=\la_\eps^{-2s}|w_\eps\|_2^2.$$
From Lemma \ref{lem-2} and Lemma \ref{lem-3}, we derive the following relation
\begin{equation}\label{rel}
 \la_\eps^{\f{-2s(q-p)}{p-2}}\|v_\eps\|_q^q=k\eps\|w_{\eps}\|_2^2\leq\f{NkS_*^{-1}}{(N-2s)(k+1)}\sigma_{\eps}(1+o(1)) \lesssim \sigma_\eps,
\end{equation}
which leads to the following lemma:

\begin{lemma}\label{lem-6}
 $\sigma_\eps^{-\f{p-2}{2s(q-p)}} \lesssim \la_\eps \lesssim \eps^{\f{-1}{2s}}\sigma_\eps^{\f{1}{2s}}.$
\end{lemma}
\begin{proof}
{\bf Claim:} $$(i)\, \liminf_{\eps \to 0}\|v_\eps\|_q>0,\quad (ii)\, \liminf_{\eps \to 0}\|v_\eps\|_2>0.$$
 To see the claim, note that by Lemma \ref{lem-5} and in view of the embedding $L^q(B_1) \subset L^p(B_1)$ we have,
 $$c\|v_\eps\chi_{B_1}\|_q \geq \|v_\eps\chi_{B_1}\|_p\geq \|W_1\chi_{B_1}\|_p-\|(W_1-v_\eps)\chi_{B_1}\|_p=\|W_1\chi_{B_1}\|_p-o(1),$$
 where $\chi_{B_R}$ is the characteristic function of $B_R.$ Similarly, in view of the embedding $L^p(B_1) \subset L^2(B_1)$, we obtain
 $$\|v_\eps\chi_{B_1}\|_2 \geq \|W_1\chi_{B_1}\|_2-\|(W_1-v_\eps)\chi_{B_1}\|_2=\|W_1\chi_{B_1}\|_2-o(1).$$
Thus the claim follows.

Combining the claim (i) with \eqref{rel}, it follows $\sigma_\eps^{-\f{p-2}{2s(q-p)}} \lesssim \la_\eps $. Similarly, combining the claim (ii) with \eqref{rel}, it yields
\be\lab{4-13-4}c\la_{\eps}^{2s}\leq \la_{\eps}^{2s}\|v_{\eps}\|^2_2=\|w_{\eps}\|_2^2\lesssim\f{\sigma_{\eps}}{\eps}.\ee
Hence the lemma follows.
\end{proof}

Using Lemma \ref{lem-1}, we infer from Lemma \ref{lem-6} a lower bound for $\la_{\eps}$
\begin{equation}\label{lower-bd}
 \la_\eps \gtrsim \sigma_\eps^{\f{-(p-2)}{2s(q-p)}} \gtrsim
   \eps^{\f{-(p-2)}{2s(q-2)}}
 \end{equation}
and an upper bound for $\la_{\eps}$
\begin{equation}\label{upper-bd}
 \la_\eps \lesssim \eps^{\f{-1}{2s}}\sigma_\eps^{\f{1}{2s}} \lesssim  
     \eps^{-\f{p-2}{2s(q-2)}}.
 \end{equation}    

Note that 
above upper estimate and lower estimate for $\la_{\eps}$ are equivalent. Consequently we obtain

\begin{corollary}\label{cor-7}
Then $\|v_\eps\|_q$ and $\|v_\eps\|_2$ are bounded.
\end{corollary}
\begin{proof}
 Note that, from \eqref{rel} and Lemma \ref{lem-1} we have
 \begin{equation}\label{**}
  \|v_\eps\|_q^q \leq c\sigma_\eps\la_\eps^{\f{2s(q-p)}{p-2}} \leq c\eps^{\f{q-p}{q-2}}\la_\eps^{\f{2s(q-p)}{p-2}}.
 \end{equation}
Hence, substituting upper estimate of $\la_{\eps}$ from \eqref{upper-bd}  we get, $\|v_\eps\|_q^q \leq c.$
Similarly, from \eqref{4-13-4}, \eqref{lower-bd} and Lemma \ref{lem-1}, we conclude
$$\|v_\eps\|_2^2 \lesssim \eps^{-1}\la_\eps^{-2s}\sigma_\eps \lesssim \eps^{-1}\eps^{\f{p-2}{q-2}}\sigma_\eps \lesssim 1.$$
 Hence, $\|v_\eps\|_2$ is bounded.
\end{proof}

\subsection{Proof of Theorem \ref{thm-2.5} concluded}

\begin{lemma}\label{lem-c}
$\|v_\eps\|_{\infty}=O(1).$
\end{lemma}
\begin{proof}
From \eqref{4-20-2}, we note that  $v_\eps$ is a positive solution of the linear inequality
 $$(-\De)^sv_\eps-V_\eps(x)v_\eps \leq 0,\,\ x \in \Rn,$$
 where $V_\eps(x):=S_\eps v_\eps^{p-2}(x).$
Using Lemma \ref{l:1}(i), we have
\be\lab{4-15-1}v_\eps(x) \leq C_q\|v_\eps\|_q|x|^{\f{-N}{q}}.\ee
 Hence, by Corollary \ref{cor-7} we obtain
 \Bea
 V_\eps(x) &\leq& S_\eps C_q^{p-2}\|v_\eps\|^{p-2}_q|x|^{\f{-N(p-2)}{q}}\\
 &\leq& C_*|x|^{-2sp/q},
 \Eea
 for some $C_*>0.$ Thus, $v_\eps$ is a positive subsolution of
\be\lab{4-17-1}(-\De)^s v -V_*(x)v= 0,\ee
 where $V_*(x)=C_*|x|^{-\f{2sp}{q}} \in L^t_{loc}(\Rn)$, for $t<\f{N-2s}{4s}q$. We choose, $t$ such that $t\in (\f{N}{2s}, \f{N-2s}{4s}q)$. Then $V_*\in L^t_{loc}(\Rn)$, where $t>\f{N}{2s}$. This in turn implies, any positive solution $v$ of \eqref{4-17-1} is in  $L^{\infty}_{loc}(\Rn)$ (see \cite[Proposition 2.4]{JLX}). Also observe that using \eqref{4-15-1} via Corollary \ref{cor-7}, we have  $v_{\eps}(x)\to 0$ uniformly as $|x|\to \infty$. Therefore, we can choose $R>0$ large enough such that
 $\max_{|x|=R}{v_{\eps}(x)}\leq \max_{|x|=R} v(x)$. Consequently, applying maximum principle $v_{\eps}\leq v$ in $B_R(0)$. This in turn implies, $\|v_{\eps}\|_{L^{\infty}(B_R(0))}$ is uniformly bounded. Combining this with \eqref{4-15-1} via Corollary \ref{cor-7}, the lemma follows.
\end{proof}

\begin{corollary}\label{cor-d}
$v_\eps \to W_1$ in $C^{2s-\de}$, for some $\de>0$ and $L^r(\Rn)$, for any $r \geq p=2^*$. In particular, $v_\eps(0)\simeq W_1(0).$
\end{corollary}
\begin{proof}
As a consequence of Lemma \ref{lem-c} and Lemma \ref{lem-5}, we obtain via interpolation inequality, $$\|v_\eps-W_1\|_r \leq \|v_\eps-W_1\|_p^{1-\theta}\|v_\eps-W_1\|_{\infty}^{\theta},$$ where $\theta$ satisfies $\f{1}{r}=\f{1-\theta}{p}$. Therefore,
$$\|v_\eps-W_1\|_r \leq C\|v_\eps-W_1\|_p^{1-\theta}\to 0.$$
We choose $r$ such that, $\max\{2^*, \f{N}{2s}\}<r<\infty$.   Then, invoking \cite[Theorem 1.6(iii)]{Ros-Ser} we have
 $$[v_\eps-W_1]_{C^{\al}(\Rn)} \leq C\|v_\eps-W_1\|_r\to 0,$$ where $\al=2s-\f{N}{r}$.
 Set, $\de:=\f{N}{r}$, which implies
 $$v_\eps \to W_1  \quad\mbox{in}\quad C^{2s-\de}(\Rn).$$
\end{proof}

\textit{\bf Proof of Theorem \ref{thm-2.5}}:
\begin{proof}
From Corollary \ref{cor-d} we have,
$v_\eps \to W_1$ in $C^{2s-\de}(\Rn)$ for some $\de>0$, $v_\eps \to W_1$ in $L^r(\Rn)$ for any $r \geq p=2^*$ and $v_\eps(0)\simeq W_1(0).$
Also from \eqref{tilde-v-eps}, we have
\be\lab{5-22-1}
\widetilde{v_\eps}(x)=\la_\eps^{\f{2s}{p-2}}u_\eps(\la_{\eps}x)=\la_\eps^{\f{2s}{p-2}}w_\eps\bigg(\f{\la_{\eps}x}{S_{\eps}^{\f{1}{2s}}}\bigg)
=v_\eps\bigg(\f{x}{S_{\eps}^{\f{1}{2s}}}\bigg).\ee
As $v_\eps \to W_1$ as $\eps \to 0,$ so $\widetilde{v_\eps} \to U_1$ in $C^{2s-\de}(\Rn)$ for some $\de>0$ and in $L^r(\Rn)$ for any $r \geq p$, where $U_1(x)=W_1\bigg(\f{x}{S_*^{\f{1}{2s}}}\bigg)$ for all $x \in \Rn.$ Further, $ \widetilde{v_\eps} \simeq U_1(0)$. Since, $p<q$, we obtain
$\widetilde{v_\eps} \to U_1$ in $L^q(\Rn)$ as $\eps \to 0.$
Also from \eqref{lower-bd} and \eqref{upper-bd}, we have
\begin{equation}\label{la-}
\la_\eps \sim \eps^{\f{-1}{2s}\f{(p-2)}{q-2}}. 
\end{equation}

Since, $u_\eps(x)=\la_\eps^{-\f{2s}{p-2}}\widetilde{v_\eps}(\f{x}{\la_{\eps}})$ and $\widetilde{v_\eps}(0)\simeq U_1(0),$ we have
$u_\eps(0)=\la_\eps^{\f{-2s}{p-2}}\widetilde{v_\eps}(0)\simeq \la_\eps^{\f{-2s}{p-2}}U_1(0).$
Moreover, using \eqref{la-}, we get $u_\eps(0) \simeq \eps^{\f{1}{q-2}}$. Hence  the theorem follows.
\end{proof}

\section{ Asymptotic behavior in the supercritical case $p>2^*$}

For $p>2^*$, the limit equation
\begin{equation}
\label{P_0}
(P_{0})\qquad
(-\De)^s u  =|u|^{p-2}u -|u|^{q-2}u \quad\text{in }\quad \Rn, 
\end{equation}
admits  a non-negative radially decreasing solution $u_0 \in \dot{H}^s(\Rn) \cap L^q(\Rn)$ (see \cite{BM}, theorem 1.7]). Further, following the similar arguments of Lemma \ref{l:max},  we obtain  $0\leq u_0\leq 1$. Therefore, since $u_0$ is a classical solution, using maximum principle it follows that $u_0$ is strictly positive in $\Rn$.

 Also, from [\cite{BM}, Theorem 1.3, Theorem 1.4] we have, $u \in C^{2s+\alpha}(\Rn)$ and 
\begin{equation}\label{E1}
\lim_{|x| \to \infty}|x|^{N-2s}u_0(x)=c_0
\end{equation}
for some $c_0>0$. Therefore, there exists $R_1>e$ such that 
\begin{equation}\label{17th Feb-1}
u_0(x)=\frac{c_0}{|x|^{N-2s}} , \quad |x|>R_1.
\end{equation}
Similarly to \eqref{u-eps}, the ground state $u_0$ admits a variational characterization in the Sobolev space $\dot{H}^s(\Rn)$ via the rescaling
$$u_0(x):=w_0\bigg(\frac{x}{S_0^{\frac{1}{2s}}}\bigg),$$
where $w_0$ is a positive radial minimizer of the constrained minimization problem
\begin{equation}\label{S_0}
S_0=\inf \bigg\{\|w\|_{\dot{H}^s(\Rn)}^2: w \in \dot{H}^s(\Rn), 2^*\int_{\Rn}F_0(w)dx=1 \bigg\},
\end{equation}
and $F_0$ is defined by
\begin{equation}
F_0(u)=\int_{0}^{u}f_0(s)ds=\begin{cases}
\frac{1}{p}u^p-\frac{1}{q}u^q,\, u \in [0,1]\\
\frac{1}{p}-\frac{1}{q},\,\, \mbox{otherwise}.
\end{cases}
\end{equation}
Similarly to \eqref{lag-mult}--\eqref{theta}, we conclude that the minimizer $w_0$ solves the following Euler-Lagrange's equation
\begin{equation}\label{E2}
(-\De)^s w_0=S_0(|w_0|^{p-2}w_0-|w_0|^{q-2}w_0) \quad\mbox{in}\quad \Rn.
\end{equation}
Further, $w_0$ satisfies the Nehari's identity
 \begin{equation}\label{E3}
 \|w_0\|_{\dot{H}^s(\Rn)}^2=S_0 \big(\|w_0\|_p^p-\|w_0\|_q^q\big)
 \end{equation}
 and the Poho\v zaev's identity
 \begin{equation}\label{E4}
 \|w_0\|_{\dot{H}^s(\Rn)}^2=S_0 2^* \bigg(\frac{1}{p}\|w_0\|_p^p-\frac{1}{q}\|w_0\|_q^q \bigg).
 \end{equation} 
 As $w_0$ is a minimizer of $(S_0),$ hence we have,
 \begin{equation}\label{E5}
 \|w_0\|_{\dot{H}^s(\Rn)}^2=S_0 \quad\mbox{with}\quad 2^*\bigg(\frac{1}{p}\|w_0\|_p^p-\frac{1}{q}\|w_0\|_q^q \bigg)=1.
 \end{equation}
Combining (\ref{E3}), (\ref{E4}) and (\ref{E5}), we obtain
 \begin{equation}\label{E6}
 \|w_0\|_p^p-\|w_0\|_q^q=\frac{2^*}{p}\|w_0\|_p^p-\frac{2^*}{q}\|w_0\|_q^q=1.
 \end{equation}
This in turn, implies
 $\|w_0\|_p^p=\frac{(q-2^*)p}{(q-p)2^*}$ and $\|w_0\|_q^q=\frac{(p-2^*)q}{(q-p)2^*}.$
 
\vspace{2mm} 
 
\subsection{ Energy and norm estimates}

To control the relations between $S_{\eps}$ and $S_0$, it is convenient to consider the (equivalent to $S_0$) scaling invariant quotient
 \begin{equation}\label{E7}
 \mathcal{S}_0(w)=\frac{\|w\|_{\dot{H}^s(\Rn)}^2}{\bigg(2^*\displaystyle\int_{\Rn}F_0(w)dx\bigg)^{\frac{N-2s}{N}}}, \quad w \in \dot{H}^s(\Rn),\quad \int_{\Rn}F_0(w)dx>0.
 \end{equation} 
 Then,
 \begin{equation}\label{E8}
 S_0=\inf_{w \in \dot{H}^s(\Rn), F_0(w)>0} \mathcal{S}_0(w).
 \end{equation}
 
 Thanks to (\ref{E1}), it follows that $w_0 \in L^2(\Rn)$ for $N>4s.$  To study the lower dimensions,  given $R>R_1$ (where $R_1$ is as defined in \eqref{17th Feb-1}), we define a cut-off function $\eta_R \in C_c^{\infty}(\R)$ as follows $\eta_R(r)=1$ for $|r|<R$, $\eta_R(r)=0$ for $|r|>2R$, $0<\eta_R(r)<1$ for $R\leq |r|\leq 2R$ and $|\eta_R'|<2/R$. 
Then using \eqref{E1},  we obtain the following three estimates:
\be\lab{8th Feb-1}\|\eta_R w_0\|_t^t=\|w_0\|_t^t\big(1-O(R^{N-t(N-2s)})\big)\quad \forall\, \, t>\frac{N}{N-2s},\ee
\begin{equation}\lab{8th Feb-3}
\|\eta_R w_0\|_2^2=\begin{cases}
             O(ln\, R),\,\quad\mbox{if}\quad N=4s,\\
             O(R^{4s-N}), \,\,\quad\mbox{if}\quad 2s<N<4s.
            \end{cases}
\end{equation}
\be\lab{8th Feb-2}\|\eta_Rw_0\|_{\dot{H}^s(\Rn)}^2 \lesssim S_0+O(R^{-(N-2s)}).\ee

 To see this,
 \Bea
 \|\eta_R w_0\|_t^t \leq\int_{B_{2R}}|w_0|^tdx&=& \int_{\Rn}|w_0|^tdx-\int_{\Rn \setminus B_{2R}}|w_0|^tdx.
 \Eea
Using \eqref{E1}, choose $R>0$  large enough so that $|w_0| \geq \frac{c_1}{|x|^{N-2s}}$ for $|x| \geq R.$
 Therefore,
 \Bea
 \int_{\Rn \setminus B_{2R}}|w_0|^tdx \geq c \int_{\Rn \setminus B_{2R}}\frac{dx}{|x|^{(N-2s)t}}\geq cR^{N-t(N-2s)}.
 \Eea
 Hence,
 \begin{equation}\label{E18}
 \|\eta_R w_0\|_t^t \leq |w_0|^t_t-\tilde{c_1}R^{N-t(N-2s)}.
 \end{equation}
similarly,  using \eqref{E1} we also have
\be\label{E19}
 \|\eta_Rw_0\|_t^t\geq \int_{B_R}|w_0|^tdx=\int_{\Rn}|w_0|^tdx-\int_{\Rn \setminus B_R}|w_0|^tdx \geq  
 |w_0|^t_t- c_2 R^{N-t(N-2s)}.
\ee
Combining (\ref{E18}) and (\ref{E19}), estimate \eqref{8th Feb-1} follows. 

Now to see \eqref{8th Feb-3}, note that
\bea \label{17th Feb-2}
\int_{\Rn}|\eta_Rw_0|^2dx=\int_{B_{2R}}\eta_R^2|w_0|^2dx&\leq& \int_{B_{2R}}|w_0|^2dx\no\\
&=& \int_{B_{R_1}}|w_0|^2dx+\int_{B_{2R}\setminus B_{R_1}}|w_0|^2dx \no\\
&\leq& \|w_0\|_{L^2(B_{R_1})}+c^2 \int_{B_{2R}\setminus B_{R_1}}\frac{dx}{|x|^{2(N-2s)}}\no\\
&\leq&\|w_0\|_{L^2(B_{R_1})}+c_1^2 \int_{1}^{2R}\frac{r^{N-1}}{r^{2(N-2s)}}dr\no\\
&=&\|w_0\|_{L^2(B_1)}+c_1^2 \int_{1}^{2R}r^{4s-N-1}dr
\eea	
On the other hand,
\bea \label{18th Feb-1}
\int_{\Rn}|\eta_Rw_0|^2dx\geq \int_{B_{R}}\eta_R^2|w_0|^2dx
=\int_{B_R}|w_0|^2dx \geq c^2\int_{B_R\setminus B_{R_1}}\frac{dx}{|x|^{2(N-2s)}}=c_2\int_{R_1}^{R}r^{4s-N-1}dr \label{+3}
\eea			
Since  $R>R_1>e>2$ implies $ln(2R)<ln (R^2)  =2 ln\, R$, in the case of $N=4s$, from \eqref{17th Feb-2} and \eqref{18th Feb-1}, we obtain
\be\lab{17th Feb-3}  c (ln\, R)<c_2 ln(R/R_1)<\int_{\Rn}|\eta_Rw_0|^2dx \leq C(1+2ln(R))\leq C(ln R),\ee
for some constant $c, C>0$.
Similarly, in this case of $2s<N<4s$, from \eqref{17th Feb-2} and \eqref{18th Feb-1}, it's straight-forward to see that
\be\lab{17th Feb-4} cR^{4s-N}<\int_{\Rn}|\eta_Rw_0|^2dx \leq C(1+R^{4s-N})\leq \tilde CR^{4s-N},\ee
for some constant $c, C>0$. Combining  \eqref{17th Feb-3} and \eqref{17th Feb-4}, estimate \eqref{8th Feb-3} follows.

{\bf Proof of \eqref{8th Feb-2}:} \bea\label{Ee}
\int_{\Rn \times \Rn} \frac{|\eta_R(x)w_0(x)-\eta_R(y)w_0(y)|^2}{|x-y|^{N+2s}}dxdy
&=&\int_{B_R \times B_R}\frac{|w_0(x)-w_0(y)|^2}{|x-y|^{N+2s}}dxdy\no\\
&+&2 \int_{B_R\times B_R^c}\frac{|\eta_R(x)w_0(x)-\eta_R(y)w_0(y)|^2}{|x-y|^{N+2s}}dxdy\no\\
&+&\int_{B_R^c \times B_R^c}\frac{|\eta_R(x)w_0(x)-\eta_R(y)w_0(y)|^2}{|x-y|^{N+2s}}dxdy\no\\
&:=&I_1 +2I_2 +I_3.
\eea	
First, we estimate $I_2.$	
\bea\label{Ef}
I_2&=& \int_{B_R\times B_R^c} \frac{|\eta_R(x)w_0(x)-\eta_R(y)w_0(y)|^2}{|x-y|^{N+2s}}dxdy\no\\
&=&\int_{B_R\times B_R^c} \frac{|w_0(x)-\eta_R(y)w_0(y)|^2}{|x-y|^{N+2s}}dxdy\no\\
&\leq&\int_{B_R\times B_R^c} \frac{|w_0(x)-w_0(y)|^2}{|x-y|^{N+2s}}dxdy+\int_{B_R\times B_R^c} \frac{|w_0(y)|^2|1-\eta_R(y)|^2}{|x-y|^{N+2s}}dxdy.\no\\
\eea
Now,
\bea\label{Eg}
\int_{B_R\times B_R^c} \frac{|w_0(y)|^2|\eta_R(x)-\eta_R(y)|^2}{|x-y|^{N+2s}}dxdy
&=&\int_{B_R^c}|w_0(y)|^2\bigg(\int_{B_R}\frac{|\eta_R(x)-\eta_R(y)|^2}{|x-y|^{N+2s}}dx  \bigg)dy\no\\
&\leq& \bigg(\int_{B_R^c}|w_0(y)|^{2^*}dy\bigg)^{\frac{2}{2^*}}.\no
\\
&&\qquad\bigg[\int_{B_R^c}\bigg(\int_{B_R}\frac{|\eta_R(x)-\eta_R(y)|^2}{|x-y|^{N+2s}}dx \bigg)^{\frac{2^*}{2^*-2}} dy\bigg]^{\frac{2^*-2}{2^*}}\no\\
&\leq&C\bigg(\int_{B_R^c}\frac{1}{|y|^{(N-2s)2^*}}dy \bigg)^{\frac{N-2s}{N}}I_{2,2}\no\\
&=&C_1R^{-(N-2s)}I_{2,2},
\eea
where
\begin{eqnarray}\label{Ee1}
I_{2,2}&=&\bigg[\int_{B_R^c}\bigg(\int_{B_R}\frac{|\eta_R(x)-\eta_R(y)|^2}{|x-y|^{N+2s}}dx  \bigg)^{N/2s}dy \bigg]^{\frac{2s}{N}}\no\\
 &\leq&\bigg[\int_{B_{2R}\setminus B_R}\bigg(\int_{B_R}\frac{|\eta_R(x)-\eta_R(y)|^2}{|x-y|^{N+2s}}dx  \bigg)^{N/2s}dy \bigg]^{\frac{2s}{N}}\no\\
&&\quad+\bigg[\int_{B_{2R}^c}\bigg(\int_{B_R}\frac{|\eta_R(x)-\eta_R(y)|^2}{|x-y|^{N+2s}}dx \bigg)^{N/2s}dy\bigg]^{\frac{2s}{N}}\no\\
&=&:(I_{2,2}^1)^{\frac{2s}{N}}+(I_{2,2}^2 )^{\frac{2s}{N}}.
\end{eqnarray} 
Now,
\bea\label{Eh}
(I_{2,2}^1)^{\frac{2s}{N}}&\leq&\bigg[\int_{B_{2R}\setminus B_R}\bigg(\int_{|x-y| \leq R}\frac{|\eta_R(x)-\eta_R(y)|^2}{|x-y|^{N+2s}}dx+\int_{|x-y|>R}\frac{|\eta_R(x)-\eta_R(y)|^2}{|x-y|^{N+2s}}dx \bigg)^{N/2s}dy\bigg]^{\frac{2s}{N}}\no\\
&\leq&
\bigg[\int_{B_{2R}\setminus B_R}\bigg(\frac{C}{R^2}\int_{|x-y| \leq R}\frac{|x-y|^2}{|x-y|^{N+2s}}dx+C\int_{|x-y|>R}\frac{1}{|x-y|^{N+2s}}dx\bigg)^{N/2s}dy  \bigg]^{\frac{2s}{N}}\no\\
&\leq&
\bigg[\int_{B_{2R}\setminus B_R}\bigg(\frac{C}{R^2}\int_{B_R}|z|^{2-N-2s}dz+C\int_{B_R^c}\frac{1}{|z|^{N+2s}}dz\bigg)^{N/2s}dy  \bigg]^{\frac{2s}{N}}\no\\
&\leq&
C\bigg(\int_{B_{2R}\setminus B_R}\bigg(\frac{R^{2-2s}}{R^22(1-s)}+\frac{R^{-2s}}{2s} \bigg)^{N/2s}dy  \bigg)^{\frac{2s}{N}}\no\\
&=&C R^{-2s}\bigg(\frac{1}{2(1-s)}+\frac{1}{2s}\bigg)|B_{2R}\setminus B_R|^{\frac{2s}{N}}\no\\
&=& C_2(s, N).
\eea

Next, to estimate $I_{2,2}^2$, we first observe that if for any two sets $U$ and $\Omega$ with $\bar\Omega\subset U$ and  dist$(U^c, \Om)=:\theta>0$, then for any $x\in\Om$ and $y\in U^c$ it holds $|x-y|\geq C_\theta|y|$. Thus,
\bea\label{Ei}
(I_{2,2}^2)^{\frac{2s}{N}}&=&\bigg(\int_{B_{2R}^c}\bigg(\int_{B_R}\frac{|\eta_R(x)-\eta_R(y)|^2}{|x-y|^{N+2s}}dx \bigg)^{N/2s}dy\bigg)^{\frac{2s}{N}}\no\\
&\leq&\bigg(\int_{B_{2R}^c}\bigg(\int_{B_R}\frac{dx}{|x-y|^{N+2s}}\bigg)^{N/2s}dy\bigg)^{\frac{2s}{N}}\no\\
&\leq&C\bigg(\int_{B_{2R}^c}\frac{1}{|y|^{(N+2s)N/2s}}|B_R|^{N/2s}dy\bigg)^{\frac{2s}{N}} \no \\
&=&CR^{N}\bigg(\int_{2R}^{\infty}r^{N-1-\frac{N}{2s}(N+2s)}dr \bigg)^{\frac{2s}{N}}\no\\
&=&C_3(N,s).
\eea
Substituting (\ref{Eh}) and (\ref{Ei}) into (\ref{Ee1}) yields
$I_{2,2} \leq C_2+C_3.$
This together with (\ref{Eg}) along with (\ref{Ef}) gives
\begin{equation}\label{Ee2}
I_2 \leq \int_{B_R\times B_R^c}\frac{|w_0(x)-w_0(y)|^2}{|x-y|^{N+2s}}+C_4R^{-(N-2s)}
\end{equation}
where $C_4=C_1(C_2+C_3).$ Next we estimate $I_3.$
\bea\label{Ej}
I_3&=&\int_{B_R^c \times B_R^c}\frac{|\eta_R(x)w_0(x)-\eta_R(y)w_0(y)|^2}{|x-y|^{N+2s}}dxdy\no\\
&\leq&\int_{B_R^c \times B_R^c}|\eta_R(x)|^2\frac{|w_0(x)-w_0(y)|^2}{|x-y|^{N+2s}}dxdy
+\int_{B_R^c \times B_R^c}|w_0(y)|^2\frac{|\eta_R(x)-\eta_R(y)|^2}{|x-y|^{N+2s}}dxdy\no\\
&\leq&\int_{B_R^c \times B_R^c}\frac{|w_0(x)-w_0(y)|^2}{|x-y|^{N+2s}}dxdy+J_3,
\eea
where 
\bea\label{Ek}
J_3&=&\int_{B_R^c \times B_R^c}|w_0(y)|^2\frac{|\eta_R(x)-\eta_R(y)|^2}{|x-y|^{N+2s}}dxdy\no\\
&\leq&\bigg(\int_{B_R^c}|w_0(y)|^{2^*}dy \bigg)^{\frac{2}{2^*}}\bigg(\int_{B_R^c}\bigg(\int_{B_R^c}\frac{{|\eta_R(x)-\eta_R(y)|^2}}{|x-y|^{N+2s}}dx \bigg)^{\frac{2^*}{2^*-2}}dy \bigg)^{\frac{2^*-2}{2^*}}\no\\
&=&C_1 R^{-(N-2s)}J_{3,1},
\eea

{\bf Claim:} $J_{3,1} \leq C_5.$ 

To see the claim,
\bea\lab{9th Mar-1}
J_{3,1} &=& \bigg[\int_{B_R^c}\bigg(\int_{B_R^c}\frac{{|\eta_R(x)-\eta_R(y)|^2}}{|x-y|^{N+2s}}dx \bigg)^{\frac{N}{2s}}dy \bigg]^{\frac{2s}{N}}\no\\
&=&\bigg[\int_{B_R^c}\bigg(\int_{B_R^c \cap \{ |x-y|\leq R \}}\frac{{|\eta_R(x)-\eta_R(y)|^2}}{|x-y|^{N+2s}}dx +\int_{B_R^c \cap \{ |x-y|> R \}}\frac{{|\eta_R(x)-\eta_R(y)|^2}}{|x-y|^{N+2s}}dx \bigg)^{\frac{N}{2s}}dy \bigg]^{\frac{2s}{N}}\no\\
&\leq& \bigg(J_{3,1}^1+J_{3,1}^2\bigg),
\eea	
where $$J_{3,1}^1=\bigg[\int_{B_R^c} \bigg( \int_{B_R^c \cap \{ |x-y|\leq R \}}\frac{{|\eta_R(x)-\eta_R(y)|^2}}{|x-y|^{N+2s}}dx \bigg)^{\frac{N}{2s}}dy\bigg]^{2s/N}$$ and
	$$J_{3,1}^2=\bigg[\int_{B_R^c} \bigg( \int_{B_R^c \cap \{ |x-y|> R \}}\frac{{|\eta_R(x)-\eta_R(y)|^2}}{|x-y|^{N+2s}}dx \bigg)^{\frac{N}{2s}}dy\bigg]^{2s/N}.$$
Since, $\eta_R(x)=0=\eta_R(y)$ for $|x|\geq 2R,\, |y|\geq 2R$ and $|\na\eta|\leq \frac{2}{R}$, we have
\bea\lab{9th Mar-2}
J_{3,1}^1&=&\bigg[\int_{B_{3R}\setminus B_R} \bigg( \int_{B_R^c \cap \{ |x-y|\leq R \}}\frac{{|\eta_R(x)-\eta_R(y)|^2}}{|x-y|^{N+2s}}dx \bigg)^{\frac{N}{2s}}dy\bigg]^{2s/N}\no\\
&\leq& C\bigg[\int_{B_{3R}\setminus B_R} \bigg( \int_{B_R^c \cap \{ |x-y|\leq R \}}\frac{1}{R^2}|x-y|^{2-N-2s}dx \bigg)^{\frac{N}{2s}}dy\bigg]^{2s/N}\no\\
&\leq&\frac{C}{R^2}\bigg[\int_{B_{3R}\setminus B_R}\bigg(\int_{|z|\leq R}|z|^{2-N-2s}dz \bigg)^{N/2s}dy \bigg]^{2s/N}\no\\
&\leq& C_2.
\eea
Similarly,
\bea\lab{9th Mar-3}
J_{3,1}^2&\leq &\bigg[\int_{B_{2R}\setminus B_R} \bigg( \int_{B_R^c \cap \{ |x-y|> R \}}\frac{{|\eta_R(x)-\eta_R(y)|^2}}{|x-y|^{N+2s}}dx \bigg)^{\frac{N}{2s}}dy\bigg]^{2s/N}\no\\
&&\quad+\bigg[\int_{B_{2R}^c} \bigg( \int_{B_R^c \cap \{ |x-y|> R \}}\frac{{|\eta_R(x)-\eta_R(y)|^2}}{|x-y|^{N+2s}}dx \bigg)^{\frac{N}{2s}}dy\bigg]^{2s/N}\no\\
&=:&J_{3,1}^{2,1}+J_{3,1}^{2,2}.
\eea
Now,
\bea\lab{9th Mar-4}
J_{3,1}^{2,1}&=&\bigg[\int_{B_{2R}\setminus B_R} \bigg( \int_{B_R^c \cap \{ |x-y|> R \}}\frac{{|\eta_R(x)-\eta_R(y)|^2}}{|x-y|^{N+2s}}dx \bigg)^{\frac{N}{2s}}dy\bigg]^{2s/N}\no\\
&\leq&\bigg[\int_{B_{2R}\setminus B_R} \bigg( \int_{\{ |x-y|> R\}}\frac{1}{|x-y|^{N+2s}}dx \bigg)^{\frac{N}{2s}}dy\bigg]^{2s/N}\no\\
&=&C_3.
\eea
Also,
\Bea
J_{3,1}^{2,2}&=&\bigg[\int_{B_{2R}^c} \bigg( \int_{B_R^c \cap \{ |x-y|> R \}}\frac{{|\eta_R(x)-\eta_R(y)|^2}}{|x-y|^{N+2s}}dx \bigg)^{\frac{N}{2s}}dy\bigg]^{2s/N}\no\\
&\leq&\bigg[\int_{B_{2R}^c} \bigg( \int_{(B_{2R}\setminus B_R) \cap \{ |x-y|> R \}}\frac{1}{|x-y|^{N+2s}}dx \bigg)^{\frac{N}{2s}}dy\bigg]^{2s/N}.\no\\
\Eea
By simple computation, it is easy to observe that $y\in B_{2R}^c$ and $x\in (B_{2R}\setminus B_R) \cap \{ |x-y|> R \}$ implies $|x-y|>\frac{1}{3}|y|$.
Therefore,
\be\lab{9th Mar-5}
J_{3,1}^{2,2}\leq \bigg[\int_{B_{2R}^c} \bigg( \int_{(B_{2R}\setminus B_R)}\frac{3^{N+2s}}{|y|^{N+2s}}dx \bigg)^{\frac{N}{2s}}dy\bigg]^{2s/N}=C_4.
\ee

Combining \eqref{9th Mar-2}--\eqref{9th Mar-5} with \eqref{9th Mar-1} proves the claim $J_{3,1}\leq C_5$.

This along with (\ref{Ek}) and (\ref{Ej}) gives 
\begin{equation}\label{El}
I_3 \leq \int_{B_R^c \times B_R^c}\frac{|w_0(x)-w_0(y)|^2}{|x-y|^{N+2s}}dxdy+C_5R^{-(N-2s)}.
\end{equation}
Therefore, substituting (\ref{Ee2}) and (\ref{El}) into (\ref{Ee}) yields
$$\|\eta_R w_0\|_{\dot{H}^s(\Rn)} \leq \int_{\Rn \times \Rn}\frac{|w_0(x)-w_0(y)|^2}{|x-y|^{N+2s}}dxdy+C_7R^{-(N-2s)},$$
that is, $\|\eta_R w_0\|_{\dot{H}^s(\Rn)}^2\leq S_0+O(R^{-(N-2s)}).$

 \begin{lemma}\label{L-1}
 	$0<S_\eps-S_0\To 0$  as $\eps\to 0$.  
	 \end{lemma}
 
 \begin{proof}
 We note that $$S_0 \leq \mathcal{S}_0(w_\eps)<\mathcal{S}_\eps(w_\eps)=S_\eps.$$	

 {\bf Case 1:} $N>4s.$  
 
 Testing $S_\eps$ against $w_0$, we have
 \Bea
 S_\eps \leq S_\eps(w_0) =\frac{\|w_0\|_{\dot{H}^s(\Rn)}^2}{\bigg(2^*\displaystyle \int_{\Rn}F_\eps(w_0)dx\bigg)^{\frac{N-2s}{N}}}
 &=& \frac{S_0}{\bigg(\frac{2^*}{p}\|w_0\|_p^p-\frac{2^*}{q}\|w_0\|_q^q-\eps\frac{2^*}{2}\|w_0\|_2^2\bigg)^{\frac{N-2s}{N}}}\\
 &=& \frac{S_0}{\bigg(1-\eps\frac{2^*}{2}\|w_0\|_2^2\bigg)^{\frac{N-2s}{N}}}\\
 &\leq& S_0 +O(\eps).
 \Eea
 This proves the lemma for $N>4s$.

{\bf Case 2:} $N=4s.$ 

Let $R=\frac{1}{\eps}.$ Testing $S_{\eps}$ against $\eta_R w_0$ and using \eqref{8th Feb-1}, \eqref{8th Feb-2} and the fact that $p>2^*=\frac{2N}{N-2s}=4,$ we obtain
  \Bea
  S_\eps \leq S_\eps(\eta_Rw_0) 
  &=& \frac{S_0+O(R^{-2s})}{\bigg(\frac{2^*}{p}\|w_0\|_p^p-\frac{2^*}{q}\|w_0\|_q^q-O(R^{4s-2ps})+O(R^{4s-2qs})-\eps O(ln R)\bigg)^{\frac{1}{2}}}\\
  &=& \frac{S_0+O(R^{-2s})}{\bigg(1-O(R^{4s-2ps})+O(R^{4s-2qs})-\eps O(ln R)\bigg)^{\frac{1}{2}}}\\
&\leq& \frac{S_0+O(R^{-2s})}{\bigg(1-O(R^{-4s})-\eps O(ln R)\bigg)^{\frac{1}{2}}}\\
&\leq& \frac{S_0+O(\eps^{2s})}{\bigg(1-O(\eps^{4s})- O(\eps ln \frac{1}{\eps})\bigg)^{\frac{1}{2}}}\\
  &\leq& S_0 +O(\eps ln \frac{1}{\eps}).
  \Eea

{\bf Case 3:} $2s<N<4s.$

In this case, it is easy to check that $p>4$.  Choosing $R=\eps^{-\frac{1}{2}}$ and testing $S_\eps$ against $\eta_Rw_0$ as before, we have
 \Bea
 S_\eps \leq S(\eta_Rw_0) &=&\frac{S_0+O(R^{-(N-2s)})}{\bigg(1-O(R^{4s-2ps})+O(R^{4s-2qs})-\eps O(R^{4s-N})\bigg)^{\frac{N-2s}{N}}}\\
 &\leq& \frac{S_0+O(R^{-(N-2s)})}{\bigg(1-O(R^{-4s})-\eps O( R^{4s-N})\bigg)^{\frac{N-2s}{N}}}\\
 &\leq& \frac{S_0+O(\eps^{\frac{N-2s}{2}})}{\bigg(1-O(\eps^{2s})- O(\eps^{\frac{N-4s}{2}+1})\bigg)^{\frac{N-2s}{N}}}\\
 &\leq& S_0 +O(\eps^{\frac{(N-2s)}{2}}).
 \Eea
This is because $2s<N<4s$ implies, $0<1-s<\frac{N-4s}{2}+1<1.$
Combining all these cases above, we have
\begin{equation}
S_\eps \leq \begin{cases}
S_0+O(\eps),\,\ N>4s\\
S_0+O(\eps ln\frac{1}{\eps}), \,\, N=4s\\
S_0+O(\eps^{\frac{N-2s}{2}}), \,\, 2s<N<4s.
\end{cases}
\end{equation} 
Hence,  $\lim_{\eps \to 0} S_\eps=S_0.$
 \end{proof}
 
 \begin{lemma}\label{L-2}
 	We have $\|w_\eps\|_{\infty} \leq 1$ and $\|w_\eps\|_t \lesssim 1$ for all $t\geq 2^*$, where $w_\eps$ is a minimizer of $(S_\eps).$
 \end{lemma}
  \begin{proof}
In view of Lemma \ref{l:2} and \eqref{u-eps}, we have
  	$$\|w_\eps\|_{\infty} =\|u_\eps\|_{\infty} \leq 1.$$ Also, by Sobolev embedding, 
  	$$\|w_\eps\|_{2^*}^2 \leq S_*^{-1}\|w_\eps\|_{\dot{H}^s(\Rn)}^2=S_*^{-1}S_\eps=S_*^{-1}(S_0(1+o(1))).$$
  	Therefore, for $t>2^*$ we have by interpolation,
  	$$\|w_\eps\|_t^t \leq \|w_\eps\|_{2^*}^{2^*} \lesssim C.$$
  	
  \end{proof}	
 
 \begin{lemma}\label{L-3}
 $\lim_{\eps \to 0} \eps\|w_\eps\|_2^2=0.$
 \end{lemma}
 \begin{proof}
 Since $w_\eps$ is a minimizer of $({S_\eps})$ we have,
 \bea\label{E9}
 1= 2^* \int_{\Rn} F_\eps(w_\eps)dx= 2^* \int_{\Rn}F_0(w_\eps)dx-\frac{2^*}{2}\eps\|w_\eps\|_2^2
 \eea
 Hence,
 \bea\label{E10}
 S_0(w_\eps)=\frac{\|w_\eps\|_{\dot{H}^s(\Rn)}^2}{\bigg(2^*\displaystyle\int_{\Rn}F_0(w_\eps)dx\bigg)^{\frac{N-2s}{N}}}= \frac{S_\eps}{\bigg(1+\frac{2^*}{2}\eps\|w_\eps\|_2^2 \bigg)^{\frac{N-2s}{N}}}
 \eea
 Suppose $\lim_{\eps \to 0}\eps \|w_\eps\|_2^2 \neq 0.$ Then, by sequential criterion of limit, there exists a sequence $(\eps_n)_{n \geq 1}$ such that $\lim_{n \to \infty} \eps_n=0$ but $\lim_{n \to \infty}\eps_n\|w_{\eps_n}\|_2^2=m>0.$
  Then by lemma \ref{L-1},  we have
  \begin{equation}\label{E11}
  \lim_{n \to \infty} S_0(w_{\eps_n})=\frac{S_0}{\bigg(1+\frac{2^*}{2}m \bigg)^{\frac{N-2s}{N}}}.
  \end{equation}
Then by Lemma \ref{L-1},
  \Bea
  S_0 \leq S_0(w_{\eps_n})=\frac{S_{\eps_n}}{\bigg(1+\frac{2^*}{2}\eps_n\|w_{\eps_n}\|_2^2  \bigg)^{\frac{N-2s}{N}}}= \frac{S_0(1+o(1))}{\bigg(1+\frac{2^*}{2}m \bigg)^{\frac{N-2s}{N}}}<S_0,
  \Eea
  which is a contradiction. Hence the lemma follows.
  \end{proof}
 
\subsection{ Proof of Theorem \ref{t:supcri}} \begin{proof}
 Let us consider a sequence $(\eps_n)_{n \geq 1}$ such that $\lim_{n \to \infty}\eps_n=0.$ Let $(w_\eps)_{\eps>0}$ be the corresponding minimizer for $(S_\eps),$. Therefore, by lemma \ref{L-1} 
 \begin{equation}\label{E13}
 \|w_{\eps_n}\|_{\dot{H}^s(\Rn)}^2=S_{\eps_n} \to S_0 \quad\mbox{as}\quad n \to \infty.
 \end{equation}
 Consequently, there exists a subsequence of $(\eps_n)$ (still denoted by $\eps_n$) such that
 \begin{equation}\label{E14}
 w_{\eps_n} \deb \bar{w} \quad\mbox{in}\quad \dot{H}^s(\Rn), \quad\text{and}\quad
  w_{\eps_n} \to \bar{w} \quad\mbox{a.e. in}\quad \Rn,
 \end{equation}	
 for some radial function $\bar{w} \in \dot{H}^s(\Rn)$.
 
 By Lemma \ref{L-2}, the sequence $(w_{\eps_n})_{n \geq 1}$ is bounded in $L^{2^*}(\Rn)$ and $L^{\infty}(\Rn).$
 
For $(\eps_n)_{n \geq 1}$ sufficiently small (that is, $n$ large enough), using Lemma \ref{l:1} and Sobolev embedding we obtain
 \bea\label{E15}
 w_{\eps_n}(x)\leq c|x|^{-\frac{(N-2s)}{2}} S_*^{-1/2}\|w_{\eps_n}\|_{\dot{H}^s(\Rn)}\leq c|x|^{-\frac{(N-2s)}{2}}S_*^{-1/2}S_0^{1/2}.
 \eea
Hence by Lemma \ref{l:1}, we conclude that
 	\begin{equation}\label{E16}
 	w_{\eps_n} \to \bar{w} \quad\mbox{in}\quad L^t(\Rn) \quad \mbox{for any}\quad t \in (2^*, \infty).
 	\end{equation}
 Further, using Lemma \ref{L-3} and \eqref{E9}, we have
 	\Bea
 	\int_{\Rn}F_0(\bar{w})dx=\lim_{n \to \infty} \int_{\Rn}F_0(w_{\eps_n})dx
 	=\lim_{n \to \infty} \bigg(\frac{1}{2^*}+\frac{\eps_n}{2}\|w_{\eps_n}\|_2^2 \bigg)
 	=\frac{1}{2^*}.
 	\Eea
 	
 	Therefore,
 	\begin{equation}\label{E17}
 	2^*\int_{\Rn}F_0(\bar{w})dx=1.
 	\end{equation}
 Thanks to weak lower semi-continuity, we also have
 	$$\|\bar{w}\|_{\dot{H}^s(\Rn)}^2 \leq \liminf_{n \to \infty} \|w_{\eps_n}\|_{\dot{H}^s(\Rn)}^2=S_0.$$
 	Hence, $\bar{w}$ is a minimizer for $S_0.$ Next, we claim that $w_{\eps_n} \to \bar{w}$ in $\dot{H}^s(\Rn).$
		
To see this, we note that as $w_{\eps_n} \deb \bar{w}$ weakly in $\dot{H}^s(\Rn),$ it follows
\Bea
\|w_{\eps_n}-\bar{w}\|^2_{\dot{H}^s(\Rn)}&=&\|w_{\eps_n}\|_{\dot{H}^s(\Rn)}^2+\|\bar{w}\|_{\dot{H}^s(\Rn)}^2-2\<w_{\eps_n},\bar{w}\>_{\dot{H}^s(\Rn)}\\
&=&S_{\eps_n}+\|\bar{w}\|_{\dot{H}^s(\Rn)}^2-2\|\bar{w}\|_{\dot{H}^s(\Rn)}^2+o(1)\\
&=&S_{\eps_n}-\|\bar{w}\|_{\dot{H}^s(\Rn)}^2+o(1)\\
&=&S_{\eps_n}-S_0+o(1)=o(1),
\Eea
where in the last line, we have used \eqref{E17} and the fact that $\bar{w}$ is a minimizer for $S_0,$ that is, $\|\bar{w}\|_{\dot{H}^s(\Rn)}=S_0.$
 Hence, $w_{\eps_n} \to \bar{w}$ strongly in $\dot{H}^s(\Rn).$	

By Sobolev embedding, $w_{\eps_n} \to \bar{w}$ strongly in $L^{2^*}(\Rn).$ As $(w_{\eps_n})$ is bounded in $L^{\infty}(\Rn),$ by interpolation we conclude $w_{\eps_n} \to \bar{w}$ in $L^t(\Rn)$ for $2 \leq t< \infty. $ From this, using elliptic regularity theory as in Corollary \ref{cor-d}, we conclude that $w_{\eps_n} \to \bar{w}$ in $C^{2s-\delta}(\Rn)$, for some $\delta>0$.
\end{proof}

\section{ Asymptotic behavior in the subcritical case  $2<p<2^*$}
{\bf Proof of Theorem \ref{t:sub}}\begin{proof}
As discussed in Section 3.2, to understand the asymptotic behavior of the ground state solution $u_{\eps}$ of $(P_{\eps})$ in the subcritical case, we consider the rescaling in \eqref{v}, which transforms  $(P_{\eps})$ to  $(\tilde{P}_{\eps})$ with the associated limit problem is given by \eqref{R-0} as $\eps\to 0$.

 Let $G_\eps:\R \to \R$ be a bounded function such that
 \begin{equation}
  G_\eps(w)=\begin{cases}
             \f{1}{p}w^p-\f{1}{2}w^2-\f{\eps^{\al}}{q}w^q,\,\quad\mbox{if}\quad w \geq 0,\\
             0, \,\,\quad\mbox{if}\quad w \leq 0.
            \end{cases}
\end{equation}
where $\al=\f{(q-2)}{s(p-2)}-1>0.$\\
Let
\begin{equation}
  \bar{G}(w)=\begin{cases}
             \f{1}{p}w^p-\f{\eps^{\al}}{q}w^q,\,\quad\mbox{if}\quad w\geq 0,\\
             0, \,\,\quad\mbox{if}\quad w \leq 0.
            \end{cases}
\end{equation}
Note that $G_\eps(w) \leq \bar{G}(w)$ for all $w \in \R.$ Also,  $\bar{G}'(w)=0$ implies $w=\eps^{\f{-\al}{q-p}}$ or $w=0$ and
\Bea
\bar{G}''(w)\big|_{w=\eps^{\f{-\al}{q-p}}}= (p-1)w^{p-2}-\eps^{\al}(q-2)w^{q-2}
= (p-q)w^{p-2}<0.
\Eea
Therefore, $\bar{G}$ attains its maximum at $w=\eps^{\f{-\al}{q-p}}.$
Hence,
$$G_\eps(w) \leq \bar{G}(w) \leq \bar{G}(\eps^{\f{-\al}{q-p}})=\bigg(\f{1}{p}-\f{1}{q}\bigg)\eps^{\f{-\al p}{(q-p)}}=C_\eps \quad\forall \ w\in\R.$$
We also note that $\bar{G}(w)\leq 0 \quad\mbox{if}\quad w \geq \big(\f{q}{p}\eps^{-\al}\big)^{\f{1}{q-p}}.$ Consequently,
$$G_\eps(w) \leq \bar{G}(w) \leq 0 \quad\mbox{if}\quad w \geq \big(\f{q}{p}\big)^{\f{1}{q-p}}\eps^{\f{-\al}{q-p}}.$$

Consider the family of constrained minimization problems
\begin{equation}\label{S-eps'}
(S_{\eps}')\left\{ \begin{aligned}
 S_\eps':=\inf\bigg\{ \|w\|_{\dot{H}^s(\Rn)}^2   : w \in H^s(\Rn),\quad 2^*\int_{\Rn}G_\eps(w)dx=1\bigg\}.\\
 \end{aligned}
 \right.
\end{equation}
Using the same technique as in the proof of Theorem \ref{t:ext}, it can be shown that the problem $(S_{\eps}')_{0 \leq \eps <\eps^*}$ are
wellposed in $H^s(\Rn)$, for some $\eps^*>0$. Towards this, we first prove the following claim:\\
{\bf Claim 1:} There exists $\eps^*>0$ such that for $\eps\in(0,\eps^*)$, the set $$\bigg\{w \in H^s(\Rn): \, 2^*\displaystyle\int_{\Rn}
G_\eps(w)dx=1\bigg\}$$ is non-empty.\\
To see this, let $G_0(w):=\f{1}{p}w^p-\f{1}{2}w^2.$ Note that $G_0(w) \to \infty$ as $w \to \infty$. i.e.,  given $M\in\R^+$, there exists $L>0$ such that $G_0(w) \geq M \quad\forall\ w \geq L$. Thus
$G_\eps(w)= G_0(w)-\f{1}{q}\eps^{\al}w^q\geq M-\f{1}{q}\eps^{\al}w^q$, for $w\geq L$.
Now, for $L\leq w \leq 2L,$ we have $G_\eps(w) \geq M-\eps^{\al}\f{2^q L^q}{q}.$ Hence for $0<\eps<\big(\f{Mq}{2^q L^q}\big)^{\f{1}{\al}}=:\eps^*$, we have $G_\eps(w)>0$ for all $w \in [L,2L].$ Taking $\zeta=\f{3L}{2}$, we get existence of $\zeta$ such that $G_\eps(\zeta)>0$ when $\eps \in (0,\eps^*)$.  Now following the same arguments as in \cite[p.325]{BL}, the proof of Claim 1 can be completed. 

Next,  following the same steps as in the proof of Proposition \ref{t:ext} by considering $G_\eps$ instead of $F_\eps$, it can be shown  that $(S_\eps')$ is well posed and $(\tilde{P}_{\eps})$ admits a  positive, radially symmetric and radially decreasing minimizer $w_{\eps}'$ for every $\eps\in[0,\eps^*)$.

Define $$S_{\eps}'(w)=\f{\|w\|_{\dot{H}^s(\Rn)}^2  }{\bigg(2^*\displaystyle\int_{\Rn}G_\eps(w)dx\bigg)^{\f{(N-2s)}{N}}}, \quad w \in \mathcal{M}_{\eps}',$$
where
$$\mathcal{M}_{\eps}':=\{0 \leq u \in H^s(\Rn):\int_{\Rn}G_\eps(w)dx >0\}.$$
Therefore,
$$S_{\eps}'=\inf_{w \in \mathcal{M_{\eps}'}}S_{\eps}'(w).$$
By well-posedness of $\{S_{\eps}'\}_{0 \leq \eps<\eps^*},$ let us denote the minimizer of $(S_{\eps}')$ as $w_{\eps}'.$
Note that,
$$
G_\eps(w_\eps')=\f{1}{p}|w_\eps'|^p-\f{1}{2}|w_\eps'|^2-\f{\eps^{\f{(q-2)}{s(p-2)}-1}}{q}|w_\eps'|^q \leq \f{1}{p}|w_\eps'|^p-\f{1}{2}|w_\eps'|^2=:G_0(w_\eps').
$$
Therefore,
$$S_0'(w_\eps'):= \f{\|w'_{\eps}\|_{\dot{H}^s(\Rn)}^2}{\bigg(2^*\displaystyle\int_{\Rn}G_0(w_0)dx\bigg)^{\f{(N-2s)}{N}}}
\leq \f{\|w'_{\eps}\|_{\dot{H}^s(\Rn)}^2 }{\bigg(2^*\displaystyle\int_{\Rn}G_\eps(w_\eps')dx\bigg)^{\f{(N-2s)}{N}}}= S_{\eps}'(w_\eps').$$
Moreover, as $w_\eps'$ is a minimizer for $S'_\eps,$ we have
\begin{equation}\label{S_0'}
 S'_0 \leq S_0'(w_\eps')\leq S'_{\eps}(w_\eps')=S'_{\eps}.
\end{equation}
Let $w_0'$ denote the corresponding minimizer for $S_0'$. By continuity, $w_0' \in \mathcal{M}'_{\eps}$ for $\eps>0$ sufficiently small. Therefore using $w_0'$ as a test function for $S_{\eps}'$, for sufficiently small $\eps>0$ we have
\Bea
S_{\eps}' \leq S_{\eps}'(w_0')&=& \f{\|w'_{0}\|_{\dot{H}^s(\Rn)}^2}{\bigg(2^*\displaystyle\int_{\Rn}G_\eps(w_0')dx\bigg)^{\f{(N-2s)}{N}}}\\
&=& \f{S_0'}{\bigg(2^*\displaystyle\int_{\Rn}G_0(w_0')dx-\f{2^*}{q}\eps^{\f{(q-2)-s(p-2)}{s(p-2)}}\|w_0'\|_q^q\bigg)^{\f{(N-2s)}{N}}}\\
&=& \f{S_0'}{\bigg(1-\f{2^*}{q}\eps^{\f{(q-2)-s(p-2)}{s(p-2)}}\|w_0'\|_q^q\bigg)^{\f{(N-2s)}{N}}}\\
&\leq& S'_0 + O\big(\eps^{\f{(q-2)-s(p-2)}{s(p-2)}}\big).\\
\Eea
Therefore, $S_{\eps}' \to S_0'$ as $\eps \to 0.$ Since $w_\eps'$ is a minimizer of $S_{\eps}'$, we have
\begin{equation}\label{1}
 1=\|w_\eps'\|_p^p-\|w_\eps'\|_2^2-\eps^{\al}\|w_\eps'\|_q^q,
\end{equation}
Moreover, following the same argument as in \eqref{Nehari} and \eqref{theta} yields\begin{equation}\label{2}
 1=\f{2^*}{p}\|w_\eps'\|_p^p-\f{2^*}{2}\|w_\eps'\|_2^2-\f{2^*\eps^\al}{q}\|w_\eps'\|_q^q,
\end{equation}
where $\al=\f{(q-2)-s(p-2)}{s(p-2)}$. Combining \eqref{1} and \eqref{2}, we obtain
$$\f{2^*}{q}-1=2^*\big(\f{1}{q}-\f{1}{p}\big)\|w_\eps'\|_p^p-2^*\big(\f{1}{q}-\f{1}{2}\big)\|w_\eps'\|_2^2,$$
and this implies,
\begin{equation}\label{3}
 \|w_\eps'\|_p^p=\f{p}{2}\f{(q-2)}{(q-p)}\|w_\eps'\|_2^2+\f{p}{2^*}\f{(q-2^*)}{(q-p)}.
\end{equation}
Using \eqref{3} and the fact that $S_{\eps}' \to S_0'$ as $\eps \to 0,$ it can be easily shown that 
$$\lim_{\eps \to 0}\|w_\eps'\|_2^2=\f{2(2^*-p)}{2^*(p-2)}, \quad \lim_{\eps \to 0}\|w_\eps'\|_p^p=\f{(2^*-2)p}{(p-2)2^*}.$$
This in turn implies, $\lim_{\eps \to 0}2^* \displaystyle\int_{\Rn}G_0(w_\eps')dx=1$.
Hence, there exists $\left(\la_\eps\right)_{\eps\geq0}$ such that $2^* \displaystyle\int_{\Rn}G_0(\tilde{w}_\eps)dx=1$
where $$\tilde{w}_\eps:=w_\eps'(\la_\eps x) \quad\text{and}\quad \lim_{\eps \to 0}\la_\eps=1.$$
Consequently, a direct computation yields \be\lab{5-16-1}\lim_{\eps \to 0}S_{\eps}'(\tilde{w}_\eps)=S_0'.\ee On the other hand, \be\lab{5-16-2}S_0'\leq S_0'(\tilde{w}_{\eps})\leq S_{\eps}'(\tilde{w}_{\eps}).\ee Combining \eqref{5-16-1} and \eqref{5-16-2}, we obtain $\lim_{\eps\to 0} S_0'(\tilde{w}_{\eps})=S_0'$, i.e.,
$\left(\tilde{w}_\eps\right)_{\eps\geq0}$ is a minimizing sequence for $S_0'$ which satisfies the constraint \be\lab{5-16-3}2^* \int_{\Rn}G_0(\tilde{w}_\eps)dx=1.\ee

\vspace{2mm}

{\bf Claim}: $w_\eps' \to w_0'$ in $H^s(\Rn)$ as $\eps \to 0.$\\

To see the claim, let $\eps_{n}\to 0$ as $n\to\infty.$ As $(\tilde{w}_\eps)$ is a minimizing sequence for $(S_0'),$
we have $\|\tilde w_{\eps_n}\|_{\dot{H}^s(\Rn)}^2\to S_0'$ as $n\to\infty$. Consequently by Sobolev inequality, there exists $C>0$ such that
$ \|\tilde{w}_{\eps_n}\|_{2^*}\leq C$, for all $n\geq 1$. Also, from \eqref{5-16-3} we have
$$ \frac{2^*}{p}\|\tilde{w}_{\eps_n}\|_p^p-\frac{2^*}{2}\|\tilde{w}_{\eps_n}\|_2^2=1,\quad \forall\ n\geq 1.$$
Therefore we have $\|\tilde{w}_{\eps_n}\|_2^2< \|\tilde{w}_{\eps_n}\|_p^p.$
Using interpolation, we conclude that
\begin{align*}
 \|\tilde{w}_{\eps_n}\|_2^2<\|\tilde{w}_{\eps_n}\|_p^p \leq
 \|\tilde{w}_{\eps_n}\|_2^{\theta p}\|\tilde{w}_{\eps_n}\|_{2^*}^{(1-\theta)p}, \quad\text{where}\quad \f{1}{p}=\frac{\theta}{2}+\f{1-\theta}{2^*}. \end{align*}
Hence 
$\|\tilde{w}_{\eps_n}\|_2\leq C$ for all $n\geq 1$. Since $\|\tilde{w}_{\eps_n}\|^2_{H^s(\Rn)}=\|\tilde{w}_{\eps_n}\|^2_{2}+[\tilde{w}_{\eps_n}]^2_{\dot{H^s}(\Rn)},$
we have $(\tilde{w}_{\eps_n})_{n\geq 1}$ is bounded in $H^s(\Rn).$
Therefore, there exists $\tilde{w} \in H^s(\Rn)$ such that $$\tilde{w}_{\eps_n} \rightharpoonup \tilde{w} \quad\text{in}\quad H^s(\Rn) \quad\text{and}\quad \tilde{w}_{\eps_n} \to \tilde{w} \quad\text{a.e. in}\,\Rn,$$ upto a subsequence. We know from \cite[Lemma 6.1]{BM},
that $H^s_{rad}(\Rn)\hookrightarrow L^t(\Rn)$ compactly for $2\leq t<2^*$.
Hence, as $\tilde{w}_{\eps_n}$ are positive, symmetric and decreasing , so $\tilde{w}_{\eps_n} \to \tilde{w}$ in $L^2(\Rn)$ strongly
as $n\to\infty.$ As $2<p<2^*,$ using interpolation, we have, $\tilde{w}_{\eps_n} \to \tilde{w}$ in $L^p(\Rn)$ as $n \to \infty.$
Therefore, we have
\be\lab{5-16-5}2^*\int_{\Rn} G_0(\tilde{w})=1.\ee
By weak lower semicontinuity of norm, we also have
\be\lab{5-16-6} \|\tilde{w}\|^2_{H^s(\Rn)} \leq\varliminf_{n\to\infty} \|\tilde{w}_{\eps_n}\|^2_{H^s(\Rn)}\quad\text{and}\quad  [\tilde{w}]^2_{\dot{H^s}(\Rn)} \leq\varliminf_{n\to\infty} [\tilde{w}_{\eps_n}]^2_{\dot{H^s}(\Rn)}.\ee
Using \eqref{5-16-5} and \eqref{5-16-6},  we have
$$S_0'\leq S_0'(\tilde{w})=\|\tilde w\|_{\dot{H}^s(\Rn)}^2
\leq \lim_{n\to\infty}\|\tilde w_{\eps_n}\|_{\dot{H}^s(\Rn)}^2=\lim_{n\to\infty} S_0'(\tilde{w}_{\eps_n})=S_0'.$$
Consequently,
$$ S_0'=S_0'(\tilde{w})= \|\tilde w\|_{\dot{H}^s(\Rn)}^2= \lim_{n\to\infty}\|\tilde w_{\eps_n}\|_{\dot{H}^s(\Rn)}^2.$$
Hence, $\lim_{n \to \infty}\|\tilde{w}_{\eps_n}\|^2_{H^s(\Rn)}=\|\tilde{w}\|^2_{H^s(\Rn)}.$ Combining this with the weak convergence
of $\tilde{w}_{\eps_n}$, we conclude that $\tilde{w}_{\eps_n} \to \tilde{w}$
strongly in $H^s(\Rn)$ as $n\to\infty$. In view of convergence of $\la_\eps$ we have,
$w_{\eps_n}' \to \tilde{w}$ in $H^s(\Rn)$ where $\tilde{w}$ is a minimizer of $S_0'.$ Therefore, By uniqueness of minimizer of \eqref{R-0}
we have, $\tilde{w}=w_0'$. Hence the claim follows.

Finally, arguing as in the proof of Lemma \ref{lem-c}, using $\|w_{\eps_n}'\|_{2^*}$ instead of $L^q$ norm to control the growth of $w_{\eps_n}' $ at the origin, we also conclude that  $\|w_{\eps_n}'\|_{\infty} \lesssim 1$ as $\eps \to 0.$
Using the above claim we have,
$$w_{\eps_n}' \to w_0' \quad\mbox{in}\quad L^2(\Rn) \quad\mbox{as}\quad n \to \infty.$$
Now, let $2<r < \infty. $ By interpolation we have,
$$
\|w_{\eps_n}'-w_0'\|_r \leq \|w_{\eps_n}'-w_0'\|_2^{\theta}\|w_{\eps_n}'-w_0'\|^{1-\theta}_{\infty}\leq \|w_{\eps_n}'-w_0'\|_2^{\theta}\big(\|w_{\eps_n}'\|_{\infty}+\|w_0'\|_{\infty}\big)^{1-\theta}\leq C\|w_{\eps_n}'-w_0'\|_2^{\theta} \to 0,$$
as $n \to \infty$, where  $\f{1}{r}=\f{\theta}{2}$.
Hence, $w_{\eps_n}' \to w_0'$ in $L^r(\Rn)$ for $2 \leq r < \infty$. Now, choose $r$ such that $\max\{2, \f{N}{2s}\}<r<\infty$.   Then, invoking \cite[Theorem 1.6(iii)]{Ros-Ser} we have
 $$[w_{\eps_n}'-w_0']_{C^{\al}(\Rn)} \leq C\|w_{\eps_n}'-w_0'\|_r\to 0,$$ where $\al=2s-\f{N}{r}$. Set, $\de:=\f{N}{r}$, which implies
 $$w_{\eps_n}' \to w_0'  \quad\mbox{in}\quad C^{2s-\de}(\Rn).$$
\end{proof}

\section{ Local uniqueness in the subcritical case $p<2^*$}
{\bf Proof of Theorem \ref{t:uni-sub}}\begin{proof}
We prove the theorem by method of contradiction. Suppose there exists a sequence $\eps_n\to 0$ and two distinct functions $u_n^1:=u_{\eps_n}^1$ and $u_n^2:=u_{\eps_n}^2$  solve $(P_{\eps})$. Now, define
$$v_n^i:=\eps_n^{-\f{1}{s(p-2)}}u_n^i(\eps_n^{\f{-1}{2s}}x), \quad i=1,2.$$ Then by the given hypothesis of the theorem,  we have
\be\lab{5-18-1}
\|v_n^i-v_0\|_{H^s(\Rn)\cap C^{2s-\si}(\Rn)}\to 0, \quad\text{as}\quad n\to\infty, \quad i=1, 2,\ee for some $\si\in(0,2s)$. Here $v_0$ is the unique positive ground state solution of \eqref{R-0}.
Define, $$w_n:=\f{v_n^1-v_n^2}{\|v_n^1-v_n^2\|_{\infty}} .$$
Therefore, $w_n$ satisfies
$$(-\De)^s w_n+w_n=\f{1}{\|v_n^1-v_n^2\|_{\infty}}[(v_n^1)^{p-1}-(v_n^2)^{p-1}]-\f{\eps_n^{\f{q-2}{s(p-2)}-1}}{\|v_n^1-v_n^2\|_{\infty}}[(v_n^1)^{q-1}-(v_n^2)^{q-1}].$$
It is easy to check that
$$(v_n^1)^{p-1}(x)-(v_n^2)^{p-1}(x)=(p-1)\int_0^1\big(t v_n^1(x)+(1-t)v_n^2(x)\big)^{p-2}(v_n^1(x)-v_n^2(x))dt.$$
Thus $w_n$ solves the following equation
\be\lab{5-18-3}
(-\De)^s w_n+w_n=\big(c_n^1(x)-\eps_n^{\f{q-2}{s(p-2)}-1}c_n^2(x)\big)w_n \quad\text{in}\quad\Rn,
\ee
where $$c_n^1(x):=(p-1)\int_0^1\big(t v_n^1(x)+(1-t)v_n^2(x)\big)^{p-2}dt$$and
$$c_n^2(x):=(q-1)\int_0^1\big(t v_n^1(x)+(1-t)v_n^2(x)\big)^{q-2}dt.$$
Moreover, as $s\in(0,1)$, we have  $p>2\implies p>s(p-2)+2$ and therefore, $q>p>s(p-2)+2$. Thus,
\be\lab{5-19-1}
c_n^1(x)\to(p-1)v_0^{p-2} \quad\text{and}\quad \eps_n^{\f{q-2}{s(p-2)}-1}c_n^2(x)\to 0.\ee
Since, $v_0$ is positive, radially symmetric and radially decreasing to 0 (see \cite{FLS}) and $v_n^i\to v_0$ in $C^{\alpha}(\Rn)$, $i=1,2$ for some $\al>0$ and $\|w_n\|_{\infty}=1$, we get $(-\De)^s w_n$ is uniformly bounded in $L^{\infty}(\Rn)$. Therefore, applying Schauder estimate \cite{RS1}, we have
\be\lab{5-19-2}
w_n\to w \,\, \text{in}\quad C^{2s-\de}_{loc}(\Rn), \quad\text{for some}\quad \de\in (0,2s).
\ee

\vspace{2mm}

{\bf Claim 1:} $\{w_n\}$ is uniformly bounded in $H^s(\Rn)$.

\vspace{2mm}

To see the claim, first we choose $\ba>0$ small such that $\ba<\f{1}{2}(p-2)^2$. Then using Sobolev inequality and \eqref{5-18-3}, we find
\bea\lab{5-19-3}
C\bigg(\int_{\Rn}w_n^{p}dx\bigg)^\f{2}{p}\leq \|w_n\|^2_{H^s(\Rn)}&=&\int_{\Rn}c_n^1(x)w_n^2dx-
\eps_n^{\f{q-2}{s(p-2)}-1}\int_{\Rn}c_n^2(x)w_n^2dx\no\\
&\leq&\int_{\Rn}c_n^1(x)w_n^{2-\ba}dx\no\\
&\leq&\bigg(\int_{\Rn}w_n^{p}dx\bigg)^\f{2-\ba}{p}\bigg(\int_{\Rn}(c_n^1(x))^\f{p}{p-2+\ba}dx\bigg)^\f{p-2+\ba}{p}.
\eea
This in turn implies \be\lab{5-19-4}
\int_{\Rn}w_n^{p}dx\leq C \bigg(\int_{\Rn}(c_n^1(x)dx)^\f{p}{p-2+\ba}\bigg)^\f{p-2+\ba}{\ba}.
\ee
Observe that, by Theorem \ref{t:rad}, we have $u_n^1$ and $u_n^2$ are radially symmetric and symmetric decreasing and so are $v_n^i, \, i=1,2$. Thus using Lemma \ref{l:1}, for $R>0$ large enough, we obtain
\be\lab{5-19-5}
\int_{|x|>R}(c_n^1(x))^\f{p}{p-2+\ba}dx\leq C\int_{|x|>R}\f{dx}{|x|^{\f{N}{2}(p-2)\f{p}{p-2+\ba}}}\leq C\int_{R}^{\infty}r^{N-1-\f{N}{2}(p-2)\f{p}{p-2+\ba}}dr<\infty.
\ee
On the other hand, as $v_n^i\to v_0$ in $C^{\alpha}(\Rn)$, $i=1,2$ for some $\al>0$ and $v_0$ is a positive, radially symmetric and radially decreasing function, it immediately follows from the definition of $c_n^1$ that
\be\lab{5-19-6}
\int_{|x|\leq R}(c_n^1(x))^\f{p}{p-2+\ba}dx\leq C
\ee
Combining \eqref{5-19-5} and \eqref{5-19-6} along with \eqref{5-19-4} we find a constants $C, C'>0$ such that
\be\lab{5-19-7}
\int_{\Rn}(c_n^1(x))^\f{p}{p-2+\ba}dx\leq C \quad\text{and}\quad \int_{\Rn}w_n^{p}\,dx\leq C'.
\ee
Plugging \eqref{5-19-7} into \eqref{5-19-3}, the claim follows.

Combining the Claim 1 along with \eqref{5-19-2}, we have $w\in H^s(\Rn)$ and $w_n\deb w$ in $H^s(\Rn)$.
On the other hand, from \eqref{5-18-3} we also have
$$\int_{\Rn}(-\De)^\f{s}{2}w_n(-\De)^\f{s}{2}\va dx+\int_{\Rn}w_n\va dx=\int_{\Rn}\big(c_n^1(x)-\eps_n^{\f{q-2}{s(p-2)}-1}c_n^2(x)\big)w_n\va dx,$$ for all $\va\in C^{\infty}_0(\Rn)$. Using the analysis done above, we can take the limit $n\to\infty$ both the sides and passing the limit we obtain
\bea\lab{5-19-8}
(-\De)^s w+w &=&(p-1)v_0^{p-2}w \quad\text{in}\quad\Rn,\no\\
w &\in& H^s(\Rn).
\eea
On the other hand, since $v_0$ is the unique ground state solution of \eqref{R-0}, invoking \cite[Theorem 3.3]{FLS},  we find that the linear space of solutions to equation \eqref{5-19-8} can be spanned by the following $N$ functions:
$$w_i(x):=\f{\pa v_0}{\pa x_i}, \quad i=1,2,\cdots, N.$$
That is, the general solution of \eqref{5-19-8} can be written as
$$w=\sum_{i=1}^N c_i\f{\pa v_0}{\pa x_i}.$$
Note that as $v_n^i$, $i=1,2$ are radially symmetric, so is $w_n$. Further, since $w_n\to w$ in $C^{2s-\de}_{loc}(\Rn)$, we conclude $w$ is radially symmetric. Therefore, $c_i=0 \quad\forall\, i=1,2,\cdots N$. This in turn implies $w_n\to 0$ in every compact subset of $\Rn$. Let $y_n\in\Rn$ such that
$$w_n(y_n)=\|w_n\|_{L^{\infty}}=1.$$
Consequently, $y_n\to\infty$ as $n\to\infty$.

\vspace{2mm}

{\bf Claim 2:} $|w_n(x)|\leq \f{C}{|x|^{N+2s}}, \quad |x|>R,$ for $n$ large enough and for some constant $C>0$ and $R>0$ independent of $n$.

\vspace{2mm}

Assuming the claim, let us first complete the proof of the theorem. From Claim 2, it follows, $w_n(y_n)\to 0$. This contradicts the fact that $w_n(y_n)=1$. Hence the  uniqueness result follows.

Now, here we prove Claim 2.  Define, $\tilde{w}_n:=\f{w_n}{\|w_n\|_2}$. Then from \eqref{5-18-3}, it follows
$$(-\De)^s\tilde{w}_n+\tilde{V}_n\tilde{w}_n=-\tilde{w}_n,$$ where
 $$\tilde{V}_n(x):= \eps_n^{\f{q-2}{s(p-2)}-1}c_n^2(x)-c_n^1(x).$$
From the definition of $c_n^1$ and $c_n^2$, it is clear that
$$|c_n^1|_{\infty}\leq(p-1)(\|v_n^1\|_{\infty}+\|v_n^2\|_{\infty})^{p-2}\leq C, $$
$$|c_n^2|_{\infty}\leq(q-1)(\|v_n^1\|_{\infty}+\|v_n^2\|_{\infty})^{q-2}\leq C, $$
for some constant $C>0$, since $v_n^i\to v_0$ in $C^{\alpha}(\Rn),\, i=1,2$. 

Thus $\tilde{V}_n$ is uniformly bounded in $L^{\infty}(\Rn)$. Using \eqref{5-19-1} and the fact that  $v_0$ is positive, radially symmetric and radially decreasing to 0 (see \cite{FLS}), we get
$\tilde{V}_n(x)\to 0$ uniformly in $n$ as $|x|\to\infty$. Therefore, it is easy to verify that for any given $\la\in(0,1)$, there exists $R>0$ (independent of $n$) such that $\tilde{V}_n(x)+\la\geq 0$, for $|x|\geq R$. Thus following the proof of \cite[Lemma C.2(i)]{FLS}, we can obtain
$$|\tilde{w}_n(x)|\leq \f{C}{1+|x|^{N+2s}},$$ where $C$ depends on only $N, s, p, q, \la, R, |{\tilde V}_n|_{\infty}$. Going back to the definition of $\tilde{w}_n$ and using Claim 1, implies
$$|w_n(x)|\leq \f{C}{|x|^{N+2s}}, \quad |x|>>1,$$
where the constant $C$ does not depend on $n$. Thus Claim 2 follows.

Hence the theorem follows.

\end{proof}

\section{ Local uniqueness in the critical case $p=2^*$}
{\bf Proof of Theorem \ref{t:uni-cri}}\begin{proof}
We prove the theorem by method of contradiction. Suppose there exists a sequence $\eps_n\to 0$ and two distinct functions $u_n^1:=u_{\eps_n}^1$ and $u_n^2:=u_{\eps_n}^2$  solve $(P_{\eps})$. Now define $\widetilde{v}_n^i:=\widetilde{v}_{\eps_n}^i$  as in \eqref{tilde-v-eps}, that is,
$$\widetilde{v}_n^i:=\la_{\eps_n}^{\f{N-2s}{2}}u_n^i(\la_{\eps_n} x), \quad i=1,2.$$ Then by the given hypothesis of the theorem,  we have
\be\lab{5-18-1}
\|\widetilde{v}_n^i-U_1\|_{\dot{H}^s(\Rn)\cap C^{2s-\si}(\Rn)}\to 0, \quad\text{as}\quad n\to \infty, \quad i=1, 2,\ee for some $\si\in(0,2s)$.
Define, $$\theta_n:=\widetilde{v}_n^1-\widetilde{v}_n^2 \quad\text{and}\quad \psi_n:=\f{\theta_n}{\|\widetilde{v}_n^1-\widetilde{v}_n^2\|_{\infty}} .$$
Corresponding to $u_{n}^i$, we define $w_{n}^i:=w_{\eps_n}^i, \, i=1,2,$ as in \eqref{u-eps} and then corresponding to $w_n^i$, we define $v_n^i:=v_{\eps_n}^i,\, i=1,2,$ as in \eqref{v-eps}. Therefore, as in \eqref{5-22-1}, we have
$$\widetilde{v}_n^i= v_n^i\bigg(\f{x}{S_{\eps_n}^\f{1}{2s}}\bigg).$$
For each $i=1,2$, $v_n^i$ satisfy \eqref{4-20-2}. Thus, $\widetilde{v}_n^i, \, i=1,2$ satisfy the following equation:
\be\lab{5-22-2}
(-\De)^s\widetilde{v}_n^i+\eps_n\la_{\eps_n}^{2s}\widetilde{v}_n^i=(\widetilde{v}_n^i)^{p-1}-\la_{\eps_n}^{\f{-2s(q-p)}{p-2}}(\widetilde{v}_n^i)^{q-1}, \quad\text{in}\quad\Rn.
\ee
Doing the computation as in the proof of Theorem \ref{t:uni-sub}, it is not difficult to check that  $\psi_n$ solves the following equation
\be\lab{5-22-3}
(-\De)^s \psi_n+\eps_n\la_{\eps_n}^{2s}\psi_n=\bigg(c_n^1(x)-\la_{\eps_n}^{-\f{2s(q-p)}{p-2}}c_n^2(x)\bigg)\psi_n \quad\text{in}\quad\Rn,
\ee
where $$c_n^1(x):=(p-1)\int_0^1\big(t \widetilde{v}_n^1(x)+(1-t)\widetilde{v}_n^2(x)\big)^{p-2}dt$$and
$$c_n^2(x):=(q-1)\int_0^1\big(t \widetilde{v}_n^1(x)+(1-t)\widetilde{v}_n^2(x)\big)^{q-2}dt.$$
Moreover, as $N>4s$, from \eqref{lower-bd} and \eqref{upper-bd} we have $\la_{\eps_n}\sim {\eps_n}^{\f{-(p-2)}{2s(q-2)}}$. Consequently,
\be\lab{5-22-4}
\eps_n\la_{\eps_n}^{2s}\to 0, \quad c_n^1(x)\to(p-1)U_1^{p-2}(x) \quad\text{and}\quad \la_{\eps_n}^{-\f{2s(q-p)}{p-2}}c_n^2(x)\to 0.\ee
Since, $\widetilde{v}_n^i\to U_1$ in $C^{\alpha}(\Rn)$, $i=1,2$ for some $\al>0$ and $\|\psi_n\|_{\infty}=1$, we get $(-\De)^s \psi_n$ is uniformly bounded in $L^{\infty}(\Rn)$ for some $n\geq n_0$. Therefore, applying Schauder estimate \cite{RS1}, we have
\be\lab{5-22-5}
\psi_n\to \psi \,\, \text{in}\quad C^{2s-\de}_{loc}(\Rn), \quad\text{for some}\quad \de\in (0,2s).
\ee

\vspace{2mm}

{\bf Claim 1:} $\{\psi_n\}$ is uniformly bounded in $\dot{H}^s(\Rn)$.

\vspace{2mm}

To see the claim, first we choose $\ba>0$ such that $\ba<\f{8s^2}{(N-2s)^2}$. Then  using Sobolev inequality and \eqref{5-22-3}, we find
\bea\lab{5-22-6}
S\bigg(\int_{\Rn}\psi_n^{2^*}dx\bigg)^\f{2}{2^*}\leq \|\psi_n\|^2_{\dot{H}^s(\Rn)}&=&-\eps_n\la_{\eps_n}^{2s}\int_{\Rn}|\psi_n|^2dx + \int_{\Rn}c_n^1(x)|\psi_n|^2dx \no\\
&\,\,&-\la_{\eps_n}^{\f{-2s(q-p)}{p-2}}\int_{\Rn}c_n^2(x)|\psi_n|^2dx\no\\
&\leq&\int_{\Rn}c_n^1(x)|\psi_n|^{2-\ba}dx\no\\
&\leq&\bigg(\int_{\Rn}\psi_n^{2^*}dx\bigg)^\f{2-\ba}{2^*}\bigg(\int_{\Rn}(c_n^1(x))^\f{2^*}{2^*-2+\ba}dx\bigg)^\f{2^*-2+\ba}{2^*}.
\eea
This in turn implies
$$\int_{\Rn}\psi_n^{2^*}dx\leq C \bigg(\int_{\Rn}(c_n^1(x))^\f{2^*}{2^*-2+\ba}dx\bigg)^\f{2^*-2+\ba}{\ba}.$$
From here following the same steps as in Claim 1 of the proof of Theorem \ref{t:uni-sub}, we can complete the proof of this claim.

Combining the Claim 1 along with \eqref{5-22-5}, we have $\psi_n\deb \psi$ in $\dot{H}^s(\Rn)$.
On the other hand, from \eqref{5-22-3} we also have
$$\int_{\Rn}(-\De)^\f{s}{2}\psi_n(-\De)^\f{s}{2}\va dx+\int_{\Rn}\eps_n\la_{\eps_n}^{2s}\psi_n\va dx=\int_{\Rn}\big(c_n^1(x)-\la_{\eps_n}^{\f{-2s(q-p)}{p-2}}c_n^2(x)\big)\psi_n\va dx,$$ for all $\va\in C^{\infty}_0(\Rn)$. Using the analysis done above, we can take the limit $n\to\infty$ both the sides and passing the limit we obtain
\bea\lab{5-22-8}
(-\De)^s \psi&=&(p-1)U_1^{p-2}\psi \quad\text{in}\quad\Rn,\no\\
\psi &\in& \dot{H}^s(\Rn).
\eea
On the other hand, as $\|\psi_n\|_{\infty}=1$ implies $\|\psi\|_{\infty}=1$ applying \cite[Theorem 1.1]{DPS},  it follows the linear space of solutions to equation \eqref{5-22-8} can be spanned by the following $N+1$ functions:
$$\psi_i(x)=\f{2x_i}{\big(1+|x|^2\big)^\f{N-2s+2}{2}}, \quad i=1,\cdots, N$$and
$$\psi_{N+1}(x)=\f{1-|x|^2}{\big(1+|x|^2\big)^\f{N-2s+2}{2}}.$$
That is, general solution of \eqref{5-22-8} can be written as
$$\psi(x)=c \f{1-|x|^2}{\big(1+|x|^2\big)^\f{N-2s+2}{2}}+\sum_{i=1}^N c_i \f{2x_i}{\big(1+|x|^2\big)^\f{N-2s+2}{2}},$$ where $c, c_i\in\R$.
Since by Theorem \ref{t:rad}, $u_n^i, \, i=1,2$ are symmetric function, so are $\widetilde{v}_n^i,\, i=1,2$ and so is $\psi$. Thus, each $c_i=0$.

\vspace{2mm}

{\bf Claim 2}:   $c=0$.

\vspace{2mm}

 Suppose $c\not=0$. We aim to get a contradiction. For simplicity of the calculation, we can take  $c=1$, that is,
\be\lab{5-22-9}\psi(x)=\f{1-|x|^2}{(1+|x|^2)^\f{N-2s+2}{2}}.\ee
As $\widetilde{v}_n^i$ satisfies \eqref{5-22-2}, applying Poho\v zaev identity \cite[Theorem A.1]{BM} to $\widetilde{v}_n^1$ and $\widetilde{v}_n^2$ and simplifying the expressions yields
\be\lab{5-24-1}
\f{s}{N}\eps_n\la_{\eps_n}^{2s}\int_{\Rn}(\widetilde{v}_n^1)^2dx=\la_{\eps_n}^{-\f{2s(q-p)}{p-2}}\bigg(\f{1}{2^*}-\f{1}{q}\bigg)\int_{\Rn}(\widetilde{v}_n^1)^q dx.
\ee
\be\lab{5-24-2}
\f{s}{N}\eps_n\la_{\eps_n}^{2s}\int_{\Rn}(\widetilde{v}_n^2)^2dx=\la_{\eps_n}^{-\f{2s(q-p)}{p-2}}\bigg(\f{1}{2^*}-\f{1}{q}\bigg)\int_{\Rn}(\widetilde{v}_n^2)^q dx.
\ee
Subtracting \eqref{5-24-2} from \eqref{5-24-1} and multiplying  $\f{1}{\|\widetilde{v}_n^1-\widetilde{v}_n^2\|_{\infty}}$ in both sides yields
\be\lab{5-24-3}
\f{s}{N}\eps_n\la_{\eps_n}^{2s}\int_{\Rn}\psi_n(\widetilde{v}_n^1+\widetilde{v}_n^2)dx=\la_{\eps_n}^{-\f{2s(q-p)}{p-2}}\bigg(\f{1}{2^*}-\f{1}{q}\bigg)\int_{\Rn}\psi_n\int_0^1\big(t\widetilde{v}_n^1+(1-t)\widetilde{v}_n^2\big)^{q-1}dtdx.
\ee
Here we observe that using \eqref{lower-bd} and \eqref{upper-bd} we have
$$\eps_n\la_{\eps_n}^{2s}\sim\eps^\f{q-p}{q-2} \quad\text{and}\quad \la_{\eps_n}^{-\f{2s(q-p)}{p-2}}\sim \eps^\f{q-p}{q-2}. $$
Moreover, as $\psi_n\to \psi$ and $\widetilde{v}_n^i\to U_1$ uniformly, for $i=1,2$, using Lemma \ref{l:1} via Claim 1 and \eqref{5-18-1}, we can pass the limit in \eqref{5-24-3}. Thus,
\be\lab{5-24-4}
\lim_{n\to\infty}\int_{\Rn}\psi_n(\widetilde{v}_n^1+\widetilde{v}_n^2)dx=2\int_{\Rn}\psi U_1 dx=2c_{N,s}\int_{\Rn}\f{1-|x|^2}{(1+|x|^2)^{N-2s+1}}dx.
\ee
Using change of variable the RHS of above equality can be computed as
\Bea
\int_{\Rn}\f{1-|x|^2}{(1+|x|^2)^{N-2s+1}}dx&=&\om_N\int_0^1\f{(1-r^2)r^{N-1}}{(1+r^2)^{N-2s+1}}dr-\om_N\int_1^0\f{(1-\f{1}{t^2})t^{-2-(N-1)}}{(1+\f{1}{t^2})^{N-2s+1}}dt\\
&=&-\om_N\int_0^1\f{r^{N-4s-1}(1-r^2)(1-r^{4s})}{(1+r^2)^{N-2s+1}}dr,
\Eea
where $\om_N$ denotes the surface measure of unit ball in $\Rn$. As $N>4s$, $$\int_0^1\f{r^{N-4s-1}(1-r^2)(1-r^{4s})}{(1+r^2)^{N-2s+1}}<\int_0^1 r^{N-4s-1}<\infty.$$ Thus,
\be\lab{5-24-5}
\lim_{n\to\infty}\int_{\Rn}\psi_n(\widetilde{v}_n^1+\widetilde{v}_n^2)dx=-C_1,
\ee
for some $C_1>0$. Similarly,
\bea\lab{5-24-6}
&&\lim_{n\to\infty}\int_{\Rn}\psi_n\int_0^1\big(t\widetilde{v}_n^1+(1-t)\widetilde{v}_n^2\big)^{q-1}dtdx \no\\
&=&\int_{\Rn}\psi U_1^{q-1}dx\no\\
&=&c_{N,s}\om_N\int_0^1\f{(1-r^2)r^{N-1}}{(1+r^2)^{1+\f{N-2s}{2}q}}dr
-c_{N,s}\om_N\int_0^1 \f{(1-r^2)r^{-1-N+(N-2s)q}}{(1+r^2)^{1+\f{N-2s}{2}q}}dr\no\\
&=&c_{N,s}\om_N\int_0^1\f{(1-r^2)r^{N-1}(1-r^{-2N+(N-2s)q})}{(1+r^2)^{1+\f{N-2s}{2}q}}dr\no\\
&=&C_2>0,
\eea
as $q>2^*=\f{2N}{N-2s}$.
Combining \eqref{5-24-5} and \eqref{5-24-6} along with \eqref{5-24-3}, we can conclude that for $n$ large enough we get $LHS$ of \eqref{5-24-3} is strictly negative and $RHS$ of  \eqref{5-24-3} is strictly positive, which is a contradiction. Hence the claim follows.

\vspace{2mm}

Claim 2 implies that $\psi\equiv 0$. Therefore, $\psi_n\to 0$ in $K$ for every compact set $K$ in $\Rn$ . Let $y_n\in\Rn$ such that
\be\lab{5-31-1}\psi_n(y_n)=\|\psi_n\|_{\infty}=1.\ee
This in turn implies $y_n\to\infty$ as $n\to \infty$.

\vspace{2mm}

{\bf Claim 3:}  $$|\psi_n(x)|\leq \f{C}{|x|^{N-2s}}, \quad\text{for}\quad  |x|>R,$$
where $C$ is independent of $n$.

\vspace{2mm}

Assuming the Claim 3, we have  $\psi_{n}(y_n)\to 0$ and this contradicts  \eqref{5-31-1}. Hence the  uniqueness result follows.

So now we are left to prove Claim 3. Note that from \eqref{5-22-3}, $\psi_n$ satisfies
\be\lab{26-2-1}
(-\De)^s\psi_n\leq c_n^1(x)\psi_n \quad\text{in}\quad\Rn.
\ee
Furthermore, for any $f\in H^{2s}(\Rn)$, the following general Kato-type inequality (see \cite[(C.8)]{FLS}) holds
  \be\lab{Kato}
  (-\De)^s |f|\leq \text{(sgn\ f)}(-\De)^s f, \quad a.e.\quad \text{on}\quad \Rn,
  \ee
  where $(\text{sgn} f)(x)=\f{f(x)}{|f(x)|}$ if $x\not=0$ and $0$ if $f(x)=0$. 

Since by given assumption $u_{\eps}^i\in H^{2s}(\Rn)$, we get  $\psi_n\in H^{2s}(\Rn)$. Thus applying \eqref{Kato} to \eqref{26-2-1} yields
\be
  (-\De)^s |\psi_n|\leq (\text{sgn}\ \psi_n)(-\De)^s \psi_n\leq\f{\psi_n}{|\psi_n|}c_n^1 \psi_n=c_n^1|\psi_n| \quad a.e.\quad \text{on}\quad \Rn.
\ee
Moreover, it is easy to see that $|\psi_n|\in H^s(\Rn)$. Now we define the Kelvin transform of  $|\psi_n|$ by $\tilde{\psi}_n$ as follows
$$\tilde{\psi}_n(x):=\f{1}{|x|^{N-2s}}|\psi_n|\big(\f{x}{|x|^2}\big), \quad x\not=0.$$
It is well-known that 
 $\tilde\psi_n\in H^s(\Rn)$ and
$(-\De)^s \tilde{\psi}_n(x)= \f{1}{|x|^{N+2s}}(-\De)^s |\psi_n|(\f{x}{|x|^2})$. Therefore, doing a straight forward computation we obtain
\begin{equation}\lab{5-19-10}
  (-\De)^s \tilde \psi_n \leq\f{1}{|x|^{4s}}c_n^1\big(\f{x}{|x|^2}\big)\tilde{\psi}_n. \end{equation}
 
Next, we show that for some $\ga>\f{N}{2s}$, $\f{1}{|x|^{4s}}c_n^1\big(\f{x}{|x|^2}\big)\in L^{\ga}_{B_r}$, for some $r>0$ small. Indeed, for $|x|<r$, $\f{x}{|x|^2}\in B_R^c$, for $R$ large and therefore using Lemma \ref{l:1} to $\widetilde{v}_n^i$, $i=1,2$ (via \eqref{5-18-1}) we have
\Bea
\int_{B_r}\f{1}{|x|^{4s\ga}}\bigg(c_n^1\big(\f{x}{|x|^2}\big)\bigg)^{\ga}&\leq&C\int_{B_r}\int_{0}^1 \f{1}{|x|^{4s\ga}}\bigg(t\widetilde{v}_n^1\big(\f{x}{|x|^2}\big)+(1-t)\widetilde{v}_n^2\big(\f{x}{|x|^2}\big)\bigg)^{(p-2)\ga} dt dx\no\\
&\leq&C\int_{B_r}\f{|x|^{-4s\ga}}{|\f{x}{|x|^2}|^{\f{N}{2}(p-2)\ga}}dx\no\\
&\leq&C\int_0^r t^{N-1-4s\ga+\f{N}{2}(p-2)\ga}dt.
\Eea
Since, $2+\f{4s}{N}<2^*=p$, we can choose, $\ga$ such that $\ga\in\bigg(\f{N}{2s}, \f{N}{4s-(p-2)\f{N}{2}}\bigg)$. For this choice of $\ga$, the RHS of above inequality is finite. Hence, $\f{1}{|x|^{4s}}c_n^1\big(\f{x}{|x|^2}\big)\in L^{\ga}_{B_r}$, for some $r>0$ small and $\ga>\f{N}{2s}$. Consequently, as $\tilde{\psi}_n\in H^s(\Rn)$ is a sub solution to
$$(-\De)^s u =\f{1}{|x|^{4s}}c_n^1\big(\f{x}{|x|^2}\big)u \quad\text{in}\quad\Rn,$$
applying Moser iteration technique to the above  equation in the spirit of \cite[Proposition 2.4]{JLX}), it is not difficult to check that
$$\sup_{B_{\f{r}{2}}}|\tilde{\psi}_n|\leq C\bigg(\int_{B_r}|\tilde{\psi}_n|^{2^*}\bigg)^\f{1}{2^*}.$$
Moreover,
$$\int_{B_r}|\tilde{\psi}_n|^{2^*}\leq \int_{\Rn}|\tilde{\psi}_n|^{2^*}= \int_{\Rn}|\psi_n|^{2^*}\leq C||\psi_n||_{\dot{H}^s(\Rn)}\leq C'.$$
The last inequality is due to Claim 1. This in turn implies,
$$|\psi_n(x)|\leq \f{C}{|x|^{N-2s}}, \quad |x|>>1,$$
for $n$ large enough and for some constant $C>0$. Hence Claim 3 follows.
\end{proof}

\vspace{3mm}

\noi{\bf Acknowledgement}: The authors would like to give a very special thanks to Prof. Vitaly Moroz whose contribution in this paper is exactly same as of any other author. The authors are greatly indebted to him. The first author is supported by the INSPIRE research grant \\
DST/INSPIRE 04/2013/000152 and the second author is supported by the
NBHM grant 2/39(12)/2014/RD-II.

\end{document}